\documentclass[a4paper,11pt]{article}
\usepackage[utf8]{inputenc}
\usepackage{amsfonts, amsmath, amssymb, amsthm, mathrsfs, mathtools}
\usepackage{geometry}
\usepackage{float}
\usepackage{varwidth}
\usepackage[shortlabels]{enumitem}
\usepackage{framed}
\usepackage[ruled,noline,noend]{algorithm2e}
\usepackage{tcolorbox}
\usepackage{graphicx, subcaption, mwe}
\usepackage{hyperref}
\hypersetup{
        colorlinks=true,   
        linkcolor=blue,    
        citecolor=blue,    
        urlcolor=blue
}

\theoremstyle{definition}
\newtheorem{theorem}{Theorem}[section]
\newtheorem{lemma}[theorem]{Lemma}
\newtheorem{proposition}[theorem]{Proposition}

\newtheorem{remark}[theorem]{Remark}

\DeclareMathOperator{\prox}{\mathbf{prox}}
\DeclareMathOperator{\Hilbert}{\mathcal{H}}
\DeclareMathOperator{\intr}{int}
\DeclareMathOperator{\dom}{dom}
\DeclareMathOperator{\range}{range}

\DeclareMathOperator*{\argmin}{argmin}

\title{Bregman Golden Ratio Algorithms\\ for Variational Inequalities}
\author{
        Matthew K. Tam\thanks{School of Mathematics \& Statistics,
                             The University of Melbourne,
                             Parkville VIC 3010, Australia.
                             Email:~\href{mailto:matthew.tam@unimelb.edu.au}
                                         {matthew.tam@unimelb.edu.au}}
    \and
    Daniel J. Uteda\thanks{School of Mathematics \& Statistics,
                             The University of Melbourne,
                             Parkville VIC 3010, Australia. Email:~\href{mailto:duteda@student.unimelb.edu.au}{duteda@student.unimelb.edu.au}}
}
\date{\today}

\begin{document}
\maketitle

\begin{abstract}
Variational inequalities provide a framework through which many optimisation problems can be solved, in particular, saddle-point problems. In this paper, we study modifications to the so-called Golden RAtio ALgorithm (GRAAL) for variational inequalities -- a method which uses a fully explicit adaptive step-size, and provides convergence results under local Lipschitz assumptions without requiring backtracking. We present and analyse two Bregman modifications to GRAAL: the first uses a fixed step-size and converges under global Lipschitz assumptions, and the second uses an adaptive step-size rule. Numerical performance of the former method is demonstrated on a bimatrix game arising in network communication, and of the latter on two problems, namely, power allocation in Gaussian communication channels and $N$-person Cournot completion games. In all of these applications, an appropriately chosen Bregman distance simplifies the projection steps computed as part of the algorithm.
\end{abstract}

\providecommand{\keywords}[1]
{
  \small	
  \noindent \textbf{Keywords.} #1
}

\section{Introduction}\label{sec:intro}

Let $\mathcal{X},\mathcal{Y}$ be real, finite-dimensional Hilbert spaces. In this work, we consider \emph{saddle-point problems} of the form
\begin{equation}\label{eq:saddle}
    \min_{x\in\mathcal{X}}\max_{y\in\mathcal{Y}} \Phi(x,y) := \psi(x) + f(x,y) - \zeta(y)
\end{equation}
where
\begin{itemize}
    \item $f\colon\mathcal{X}\times\mathcal{Y}\to\mathbb{R}$ is convex-concave and continuously differentiable.\label{ass:f}
    \item $\psi\colon\mathcal{X}\to(-\infty,+\infty],\zeta\colon\mathcal{Y}\to(-\infty,+\infty]$ are proper,  lower semicontinuous (l.s.c.), and convex.
\end{itemize}
Since many non-smooth optimisation problems can be cast in the form of~\eqref{eq:saddle}, it is a useful and heavily studied tool in itself \cite{adolphs2019local, DURR2013540, akimoto2021saddle, lyashko2011low, nesterov2007dual, hamedani2018optimization}. However, rather than attempt to solve~\eqref{eq:saddle} in its current form, it is far more convenient, even within the afrorementioned references, to follow in the steps of Korpelevi\v{c} \cite{korpelevich1976extragradient} and Popov \cite{popov1980modification}, by casting it as the \emph{Variational Inequality (VI)}:
\begin{equation}\label{eq:VI}
    \text{Find $z^*\in\Hilbert$ such that~} \langle F(z^*),z-z^*\rangle + g(z) - g(z^*) \geq 0 \quad \forall z\in\Hilbert,
\end{equation}
where
\begin{equation}\label{eq:saddle to VI}
    z=(x,y)\in\Hilbert:=\mathcal{X}\oplus\mathcal{Y},\quad F(z) = (\nabla_x f(x,y), -\nabla_y f(x,y)), \quad g(z) = \psi(x) + \zeta(y),
\end{equation}
and the variable $z^*:=(x^*,y^*)$ shown in~\eqref{eq:VI} characterises the solution $(x^*,y^*)$ to~\eqref{eq:saddle}.

Many methods (see, for instance, \cite{Pang1982IterativeMF, marcotee1995projection, fukishima1992equivalent}) for solving~\eqref{eq:VI} require global Lipschitz continuity of the operator $F$. However, this assumption is often too strong to hold in practice. Even when $F$ is globally Lipschitz continuous, knowledge of its Lipschitz constant is usually required as input to the chosen algorithm and determining this constant is typically more difficult than solving the original problem. Moreover, even if $F$ is globally Lipschitz and its global Lipschitz constant is known, then, as the step-size is inversely related to the Lipschitz constant, a constant step-size rule can be too conservative. This is particularly unnecessary if the generated sequence lies entirely within a region where a \emph{local} Lipschitz constant is small (relative to the size of the global constant).

Therefore, it is beneficial to instead define a step-size sequence which attempts to approximate a local Lipschitz constant with respect to the point iterates. The standard approach then is to generate a step-size sequence via a \emph{backtracking} procedure (see \cite{malitsky2020forward, bellocruz2015variant, censor1998interior, malitsky2018proximal, iusem1997variant, malitsky2020projected} and references therein). While avoiding each of the shortcomings listed above, such methods can become expensive when considering the overall run-time of the algorithm, due to the arbitrarily large number of steps taken during the backtracking procedure within each iteration. An emerging alternative is that of \emph{adaptive step-sizes} \cite{malitsky2020golden, malitsky2020adaptive, alacoglu2020convergence}, which accomplish the same goals as backtracking methods without the need for backtracking, ie, the step-size update is fully explicit. In particular, the \emph{adaptive Golden RAtio ALgorithm (aGRAAL)} \cite{malitsky2020golden} (as stated in Algorithm~\ref{alg:aGRAAL}), named as such because of it's relationship with the \textit{Golden Ratio} $\varphi=\frac{1+\sqrt{5}}{2}$, solves~\eqref{eq:VI}, and is the method we focus on here. 

\begin{algorithm}[!htb]
\caption{The Adaptive Golden RAtio ALgorithm (aGRAAL) \cite{malitsky2020golden}.}\label{alg:aGRAAL}
    \SetKwInput{Init}{Initialisation}
    \KwIn{Initial points $z_0,\overline{z}_0\in\Hilbert$, initial step-size $\lambda_0>0$, $\lambda_{\max}\gg 0$, $\phi\in(1,\varphi]$.}
    \Init{Set $z_1=\overline{z}_0$, $\theta_0=1$, $\rho=\frac{1}{\phi}+\frac{1}{\phi^2}$.}
    \For{$k=1,2,\dots$}{
    Calculate the step-size:
    \begin{equation}
        \lambda_k = \min\left\{\rho\lambda_{k-1}, \frac{\phi\theta_{k-1}}{4\lambda_{k-1}}\frac{\|z_k - z_{k-1}\|^2}{\|F(z_k) - F(z_{k-1})\|^2}, \lambda_{\max}\right\}.
    \end{equation}\\
    Compute the next iterates:
    \begin{align}\\
    \overline{z}_k &= \frac{(\phi - 1)z_k + \overline{z}_{k-1}}{\phi}\\
    z_{k+1} &= \prox_{\lambda_k g}(\overline{z}_k-\lambda_k F(z_k))\label{eq:aGRAAL prox}
    \end{align}
    Update: $\theta_k = \frac{\lambda_k\phi}{\lambda_{k-1}}.$
}
\end{algorithm}

One other way to potentially improve methods is to replace the Euclidean distance in the proximal operator with a non-Euclidean family of distance-like functions called the \emph{Bregman distance}. Such methods, for solving~\eqref{eq:VI}, can be found in existing literature \cite{censor1998interior, nomirovskii2019convergence, vanhieu2022modified, gibali2018new, gibali2020fast, sym12122007, jolaoso2020weak, jolaso2021single, vanhieu2020two}. Interestingly, most of these methods require a Lipschitz assumption but don't require knowledge of the Lipschitz constant. However, these employ a backtracking procedure and/or a non-increasing step-size sequence, whereas the step-size of our new method is fully explicit and is allowed to increase slightly at each iteration. 

In this paper, we investigate Bregman modifications to Algorithm~\ref{alg:aGRAAL}. To this end, we begin by proposing the \textit{Bregman-Golden RAtio ALgorithm (B-GRAAL)}, a Bregman version of the fixed step-size \textit{Golden RAtio ALgorithm (GRAAL)}, and prove convergence of our new method in full. We then present an adaptive version of B-GRAAL, or similarly, a Bregman modification to Algorithm~\ref{alg:aGRAAL}, which we refer to as the \textit{Bregman-adaptive Golden RAtio ALgorithm (B-aGRAAL)}. Although we only provide a convergence analysis of B-aGRAAL in a restrictive setting, we observe it to work numerically outside this setting.

One advantage of our new method is the flexibility provided by the Bregman proximal operator. In the context of convex-concave games, for instance, this modified operator arises as the projection onto the probability simplex which has a simple closed form expression with respect to the \emph{Kullback--Leibler (KL) divergence} but not with respect to the standard Euclidean distance. In fact, the Euclidean projection requires an $O(n\log n)$ time algorithm in $n$-dimensions \cite{wang2012projection, chen2011projection, dai2022distributed}, whereas the KL projection only requires $O(n)$ time (see, for instance \cite[Section 5]{BECK2003167} and \cite[Section 4.4]{NIPS2015_f60bb6bb}). Another advantage of these modifications in the constrained optimisation case, is that it is sometimes possible to choose a Bregman distance whose domain is the constraint set so as to make the feasibility of the iterates implicit.

The remainder of this paper is structure as follows. In Section~\ref{sec:prelim}, we collect preliminary results for use in our analysis. In Section~\ref{sec:bgraal}, we present our fixed step-size method and a proof of convergence. In Section~\ref{sec:bagraal}, we present our adaptive method with some partial analysis. Section~\ref{sec:exp} contains experimental results. Firstly, we compare the fixed step-size method with the Euclidean distance and the KL divergence on a matrix game between two players. Secondly, we make the same comparison for the adaptive method on power allocation problem in a Gaussian communication channel. Finally, we apply the adaptive method to an $N$-person Cournot oligopoly model with appropriately chosen Bregman distances over a closed box. We then conclude this paper by presenting some directions for further research.

\section{Preliminaries}\label{sec:prelim}
Throughout this work, $\Hilbert$ denotes a real, finite-dimensional Hilbert space with inner-product $\langle\cdot,\cdot\rangle$ and induced norm $\|\cdot\|$. Given an extended real-valued function $f\colon\Hilbert\to(-\infty,+\infty]$, its \emph{domain} is denoted $\dom f := \{x\in\Hilbert\colon f(x)<+\infty\}$. Its \emph{subdifferential} at $x\in\dom f$ is given by
$$\partial f(x) := \{\nu\in\Hilbert\colon f(x) - f(y) - \langle\nu,x-y\rangle\leq 0\quad\forall y\in\Hilbert\},$$
and defined as $\partial f(x):=\emptyset$ for $x\not\in\dom f$. The \textit{indicator function} of a set $K\subseteq\Hilbert$ is written $\iota_K$ and takes the value $0$ for $x\in K$ and $+\infty$ otherwise.

A proper, l.s.c., convex function $h\colon\Hilbert\to(-\infty,+\infty]$ is called \textit{Legendre}~\cite[Definition~7.1.1]{borwein2010convex} if it is strictly convex on every convex subset of $\dom\partial h:=\{x\in\Hilbert\colon\partial h(x)\neq\emptyset\}$ and differentiable on $\intr\dom h\neq\emptyset$ such that $\|\nabla h(x)\|\to\infty$ whenever $x$ approaches the boundary of $\dom h$. The \textit{convex conjugate} of $h$ written as $h^*\colon\Hilbert\to(-\infty,+\infty]$ is the given by
$$h^*(x^*) := \sup_{x\in\Hilbert}\{\langle x^*,x\rangle - h(x)\}.$$
When $h$ is merely differentiable on $\intr\dom h$, the \emph{Bregman distance} generated by $h$ is the function $D_h\colon \Hilbert\times\intr\dom h\to(-\infty,+\infty]$ given by
$$D_h(x,y) := h(x) - h(y) - \langle\nabla h(y),x-y\rangle.$$
When $h$ is also convex, $D_h$ is non-negative and, when $h$ is $\sigma$-strongly convex, $D_h$ satisfies $$ D_h(x,y)\geq\frac{\sigma}{2}\|x-y\|^2\quad \forall (x,y)\in \Hilbert\times\intr\dom h.$$
We begin by collecting some general properties of the Bregman distance.
\begin{proposition}[Properties of the Bregman distance]\label{prop:prop}
Let $h\colon\Hilbert\to(-\infty,+\infty]$ be Legendre. Then the following assertions hold.
\begin{enumerate}[(a)]
    \item\label{item:3} \emph{(three point identity)} 
    For all $x,y\in\intr\dom h$ and $z\in\dom h$, we have
    $$D_h(z,x) - D_h(z,y) - D_h(y,x) = \langle\nabla h(x) - \nabla h(y), y-z\rangle.$$
    \item\label{item:dual} For all $x,y\in\intr\dom h$, we have
    $$D_h(x,y) = D_{h^*}(\nabla h(y), \nabla h(x))$$
    \item\label{item:identity} 
    For all $x\in\dom h$ and  $y,u,v\in\intr\dom h$ such that $\nabla h(y) = \alpha\nabla h(u) + (1-\alpha)\nabla h(v)$ for some $\alpha\in\mathbb{R}$, we have
    $$D_h(x,y) = \alpha\Big[D_h(x,u) - D_h(y,u)\Big] + (1-\alpha)\Big[D_h(x,v) - D_h(y,v)\Big].$$
\end{enumerate}
\end{proposition}
\begin{proof}
\ref{item:3}~See, for instance, \cite[Lemma 2.2]{teboulle2018simplified} and the paragraph immediately after.
\ref{item:dual}~See, for instance, \cite[Theorem 3.7(v)]{bauschke1997legendre}.
\ref{item:identity}~By using the definition of $D_h$, together with the assumption $\nabla h(y) = \alpha\nabla h(u) + (1-\alpha)\nabla h(v)$, we obtain
    \begin{align*}
        D_h(x,y) &= h(x) - h(y) - \langle \nabla h(y),x-y\rangle\\
        &= \alpha\Big[h(x) - h(y) - \langle\nabla h(u),x-y\rangle\Big] + (1-\alpha)\Big[h(x) - h(y) - \langle\nabla h(v),x-y\rangle\Big]\\
        &= \alpha\Big[h(x) - h(u) - \langle\nabla h(u), x-u\rangle - h(y) + h(u) - \langle\nabla h(u),u-y\rangle\Big]\\
         & \qquad + (1-\alpha)\Big[h(x) - h(v) - \langle\nabla h(v), x-v\rangle - h(y) + h(v) - \langle\nabla h(v),v-y\rangle\Big]  \\
        &= \alpha\Big[D_h(x,u) - D_h(y,u)\Big] + (1-\alpha)\Big[D_h(x,v) - D_h(y,v)\Big].
    \end{align*}
This completes the proof.
\end{proof}
\begin{remark}
When $h=\frac{1}{2}\|\cdot\|^2$, Proposition~\ref{prop:prop}\ref{item:identity} recovers the established Euclidean identity
$$\forall x,y\in\Hilbert, \alpha\in\mathbb{R} \qquad \|\alpha x + (1-\alpha)y\|^2 = \alpha\|x\|^2 + (1-\alpha)\|y\|^2 - \alpha(1-\alpha)\|x-y\|^2.$$
\end{remark}
We now turn our attention to operators defined in terms of the Bregman divergence. The \emph{(left) Bregman proximal operator} of a function $f\colon\Hilbert\to(-\infty,+\infty]$ is the (potentially set-valued) operator given by 
\begin{equation}\label{eq:breg prox def}
\prox_f^h(y) := \argmin_{x\in\Hilbert}\left\{f(x) + D_h(x,y)\right\}\quad\forall y\in\intr\dom h.
\end{equation}
Since we will only require the left Bregman proximal operator in this work, will omit the qualifier ``left'' from here on-wards. For further details on the analogous ``right Bregman proximal operator'', the reader is referred to \cite{borwein2011characterization}. The \emph{(left) Bregman projection} onto $C$ is the (left) Bregman proximal operator of $\iota_C$, that is,
$$P^h_C(y) := \prox_{\iota_C}^h(y) = \argmin_{x\in C} D_h(x,y)\quad\forall y\in\intr\dom h.$$

Next, we collect properties of the Bregman proximal operator for use in our subsequence algorithm analysis.

\begin{proposition}[Bregman proximal operator]\label{prop:prox}
Let $f\colon\Hilbert\to(-\infty,+\infty]$ be proper, l.s.c, convex and let $h\colon\Hilbert\to(-\infty,+\infty]$ be Legendre such that $\intr\dom h\cap\dom f\neq\emptyset$.
\begin{enumerate}[(a)]
    \item\label{item:range}  $\range(\prox^h_f)\subseteq\intr\dom h\cap\dom f$.
    \item\label{item:sv} $\prox^h_f$ is single-valued on $\dom (\prox^h_f)\subseteq\intr\dom h$. Moreover, if $h+f$ is supercoercive, then $\dom (\prox^h_f)=\intr\dom h$.
    \item\label{item:prox thm} Let $y\in\dom(\prox^h_f)$ and $x=\prox_f^h(y)$. Then, for all $u\in\Hilbert$, we have
    \begin{equation}\label{eq:prox thm}
    f(x) - f(u) \leq \langle\nabla h(y) - \nabla h(x),x-u\rangle. 
    \end{equation}
    \item\label{item:bfne} 
    Let $y,y'\in\dom(\prox^h_f)$, $x=\prox_f^h(y)$ and $x'=\prox_f^h(y')$. Then
    \begin{equation}\label{eq:bfne}
    0\leq\langle \nabla h(x) - \nabla h(x^\prime), x-x^\prime\rangle \leq \langle \nabla h(y) - \nabla h(y^\prime), x-x^\prime\rangle.
    \end{equation}
\end{enumerate}
\end{proposition}
\begin{proof}
\ref{item:range}~See \cite[Proposition~3.23(v)(b)]{bauschke2003bregman}, noting that the sum rule holds for $f$ and $h$ since $\intr\dom h\cap\dom f\neq\emptyset$ and thus $\dom \partial\left(f+h\right) = \dom\left(\partial f + \nabla h\right) \subseteq \dom f\cap\intr\dom h$.
\ref{item:sv}~The first part follows by combining (a) and \cite[Proposition 3.22 (ii)(d)]{bauschke2003bregman}. For the second part, see \cite[Proposition~3.21(vii)]{bauschke2003bregman}.
\ref{item:prox thm}~The first order optimality condition together with (a) implies $\nabla h(y)-\nabla h (x)\in\partial f(x)$ which establishes the inequality. \ref{item:bfne}~The first inequality in \eqref{eq:bfne} follows from convexity of $h$. To show the second, we apply~\eqref{eq:prox thm} with $u=x^\prime$ to see
$$f(x)-f(x^\prime)\leq\langle\nabla h(y)-\nabla h(x),x-x^\prime\rangle,$$
and similarly,
$$f(x^\prime) - f(x)\leq\langle\nabla h(y^\prime) - \nabla h(x^\prime),x^\prime-x\rangle.$$
Then adding these inequalities gives the desired result.
\end{proof}
\begin{remark}
Parts \ref{item:range}, \ref{item:sv}, \ref{item:bfne} of Proposition~\ref{prop:prox} also apply to the \emph{Bregman resolvent} \cite{borwein2011characterization}, which is defined for a set-valued operator $A\colon\Hilbert\rightrightarrows\Hilbert$ as $R^h_{A} = \left(\nabla h + A\right)^{-1}\circ\nabla h$. To see that the resolvent generalises the proximal operator, we refer to \cite[Proposition 3.22 (ii)(a)]{bauschke2003bregman}.
\end{remark}

\section{The Bregman-Golden Ratio Algorithm}\label{sec:bgraal}

In this section, we consider the \textit{Variational Inequality (VI)} problem:
\begin{equation}\label{eq:VI2}
    \text{Find~}z^*\in\Hilbert\text{~such that~}\langle F(z^*), z-z^*\rangle + g(z) - g(z^*) \geq 0 \quad \forall z\in\Hilbert,
\end{equation}
where we assume that 
\begin{enumerate}[\textbf{A.\arabic*}]
    \item \label{ass:g} $g\colon\Hilbert\to(-\infty,+\infty]$ is proper, l.s.c., convex.
    \item\label{ass:h} $h\colon\Hilbert\to(-\infty,+\infty]$ is continuously differentiable, Legendre, and $\sigma$-strongly convex. In addition, we will also require that $D_h(x,x_n)\to 0$ for every sequence $(x_n)\subseteq\intr\dom h$ that converges to some $x\in\dom h$.
    \item \label{ass:F} $F\colon\Hilbert\to\Hilbert$ is monotone over $\dom g\cap\intr\dom h\neq\emptyset$.    
    \item \label{ass:S} $\Omega:=S\cap\dom h \neq \emptyset$ where $S$ denotes the solution set of~\eqref{eq:VI2}.
\end{enumerate}

\begin{remark}
Assumption~\ref{ass:h} is common in the literature concerning Bregman first-order methods~\cite{teboulle2018simplified, Chen1993ConvergenceAO, bauschke2017descent, bauschke2019linear}. In particular, the limit condition holds when $\nabla h$ is continuous and $\dom h$ is open. Indeed, in this case, $\sigma$-strong convexity of $h$ implies that $h^*$ is $\frac{1}{\sigma}$-smooth \cite[Theorem 6]{kakade2009duality}, and so applying Proposition~\ref{prop:prop}\ref{item:dual} gives
$$D_h(x,x_n)=D_{h^*}\left(\nabla h(x_n), \nabla h(x)\right) \leq \frac{1}{2\sigma}\|\nabla h(x_n) - \nabla h(x)\|^2 \to 0.$$
The significance of $\dom h$ being open here is that $h$ is differentiable at $x\in\dom h$, however we observe that the same condition can still hold if $\dom h$ is closed. In particular, it also holds for the KL divergence (see, for instance \cite[Example 2.1]{Chen1993ConvergenceAO}).
\end{remark}

Our proposed algorithm for \eqref{eq:VI2}, the \emph{Bregman-Golden RAtio Algorithm (B-GRAAL)}, is stated in Algorithm~\ref{alg:B-GRAAL}. Once again, $\varphi:=\frac{1+\sqrt{5}}{2}$ denotes the \emph{Golden Ratio}, which satisfies $\varphi^2=\varphi+1$.

\begin{algorithm}[!htb]
\caption{The Bregman-Golden RAtio ALgorithm (B-GRAAL)}\label{alg:B-GRAAL}
    \KwIn{Initial points $z_1,\overline{z}_0\in\intr\dom h$ and a step-size $\lambda \in \left(0,\frac{\sigma\varphi}{2L}\right]$.}
    \For{$k=1,2,\dots$}{
    Compute the next iterates:
    \begin{align}
        \overline{z}_k &= (\nabla h)^{-1}\left(\frac{(\varphi-1)\nabla h(z_k) + \nabla h(\overline{z}_{k-1})}{\varphi}\right) \label{eq:breg-h}\\
        z_{k+1} &=\argmin_{z\in\Hilbert}\left\{\langle F(z_k), z-z_k\rangle + g(z) + \frac{1}{\lambda} D_h(z,\overline{z}_k)\right\}.\label{eq:breg-prox}
    \end{align}
        }
\end{algorithm}

The following lemma establishes the well-definedness of the sequences generated by the Bregman-GRAAL.
\begin{lemma}\label{lem:contained}
Suppose Assumptions~\ref{ass:g}-\ref{ass:h} hold. Then the sequences $(\overline{z}_k)$ and $(z_k)$ generated by Algorithm~\ref{alg:B-GRAAL} are well-defined. Moreover, $(\overline{z}_k)\subseteq\intr\dom h$ and  $(z_k)\subseteq\intr\dom h\cap \dom g$.
\end{lemma}
\begin{proof}
Suppose by way of induction that $z_{k},\overline{z}_{k-1}\in\intr\dom h$ for some $k\geq 1$.
Since the gradient $\nabla h\colon\intr\dom h\to\intr(\dom h^*)$ is a bijection~\cite[Theorem~7.3.7]{borwein2010convex}, it follows that $\nabla h(z_k),\nabla h(\overline{z}_{k-1})\in\intr(\dom h^*)$. As $\intr(\dom h^*)$ is a convex set, \eqref{eq:breg-h} implies that  $\nabla h(\overline{z}_k) \in \intr(\dom h^*)$ which establishes that $\overline{z}_k\in\intr\dom h$.
Next, we observe that $z_{k+1} = \prox_{\lambda f}^h(\overline{z}_k)$, where  $f(z):=\langle F(z_k), z-z_k\rangle + g(z)$. Note that $\dom f=\dom g$. Since $\lambda f+h$ is $\sigma$-strongly convex, it is supercoercive by \cite[Corollary~11.16]{bauschke2011convex} and so Proposition~\ref{prop:prox}\ref{item:range}-\ref{item:sv} shows that $\prox^h_{\lambda f}$ is single-valued with $\range(\prox^h_{\lambda f})\subseteq\intr\dom h\cap\dom g$ and therefore $z_{k+1}\in\intr\dom h\cap\dom g$.
\end{proof}

\begin{remark}
The Bregman proximal step shown in~\eqref{eq:breg-prox} can be expressed in terms of the Bregman proximal operator: $z_{k+1}=\prox_{\lambda f}^h(\overline{z}_k)$, where $f(z) = \langle F(z_k), z-z_k\rangle + g(z)$. Equivalenty, $z_{k+1} = \prox_{\lambda g}^h(\left(\nabla h\right)^{-1}(\nabla h(\overline{z}_k) - \lambda F(z_k)))$, due to the first-order optimality condition in Proposition~\ref{prop:prox}\ref{item:prox thm}.
\end{remark}

The following lemma is key in our convergence analysis.
\begin{lemma}\label{lem:dec}
Suppose Assumptions~\ref{ass:g}-\ref{ass:S} hold and that $F$ is $L$-Lipschitz continuous on $\dom g\cap\intr\dom h$. Let $z^*\in \Omega$ be arbitrary. Then the sequences $(z_k), (\overline{z}_k)$ generated by Algorithm~\ref{alg:B-GRAAL} satisfy
\begin{multline}\label{eq:energy}
0 \leq (\varphi+1)D_h(z^*, \overline{z}_{k+1}) + \frac{\varphi}{2}D_h(z_{k+1}, z_k) - \varphi D_h(z_{k+1}, \overline{z}_{k+1}) \\
\leq (\varphi+1)D_h(z^*, \overline{z}_k) + \frac{\varphi}{2}D_h(z_k, z_{k-1}) - \varphi D_h(z_k, \overline{z}_k) - \left(1-\frac{1}{\varphi}\right)D_h(z_{k+1}, \overline{z}_k).
\end{multline}
\end{lemma}

\begin{proof}
By first applying Proposition~\ref{prop:prox}\ref{item:prox thm} with $f(z) := \lambda(\langle F(z_k), z-z_k\rangle + g(z))$, $u:=z\in\dom h\cap \dom g$ arbitrary, $x:=z_{k+1}$ and $y:=\overline{z}_k$, followed by the three point idenitity (Proposition~\ref{prop:prop}\ref{item:3}) we obtain
\begin{equation}\label{eq:prox-ineq}
\begin{aligned}
\lambda\left(\langle F(z_k),z_{k+1}-z\rangle + g(z_{k+1}) - g(z)\right) &\leq \langle \nabla h(\overline{z}_k) - \nabla h(z_{k+1}), z_{k+1} - z\rangle\\
&= D_h(z, \overline{z}_k) - D_h(z, z_{k+1}) - D_h(z_{k+1}, \overline{z}_k).
\end{aligned}
\end{equation}
Shifting the index in~\eqref{eq:prox-ineq} (by setting $k\equiv k-1$), setting $z:=z_{k+1}$, and using $\varphi\nabla h(\overline{z}_k) = (\varphi - 1)\nabla h(z_k) + \nabla h(\overline{z}_{k-1})$ followed by the three point identity (Proposition~\ref{prop:prop}\ref{item:3}) gives
\begin{equation}\label{eq:shift prox-ineq}
\begin{aligned}
    &\lambda\left(\langle F(z_{k-1}), z_k - z_{k+1}\rangle + g(z_k) - g(z_{k+1})\right) \\
    &\quad \leq \langle \nabla h(\overline{z}_{k-1}) - \nabla h(z_k), z_k - z_{k+1}\rangle\\
   &\quad= \varphi\langle\nabla h(\overline{z}_k) - \nabla h(z_k), z_k - z_{k+1}\rangle\\
   &\quad =\varphi\left[ D_h(z_{k+1}, \overline{z}_k) - D_h(z_{k+1}, z_k) - D_h(z_k, \overline{z}_k)\right].
\end{aligned}
\end{equation}
Let $z^*\in S\cap\dom h$. Setting $z:=z^*$ in~\eqref{eq:prox-ineq}, summing with~\eqref{eq:shift prox-ineq} and rearranging yields
\begin{equation}\label{eq:add prox-ineqs}
    \begin{aligned}
    \lambda\left(\langle F(z_k), z_k - z^*\rangle + g(z_k) - g(z^*)\right) \leq
 & D_h(z^*,\overline{z}_k) - D_h(z^*,z_{k+1}) - D_h(z_{k+1},\overline{z}_k)\\
    &+ \varphi\Big[D_h(z_{k+1},\overline{z}_k) - D_h(z_{k+1},z_k) - D_h(z_k,\overline{z}_k)\Big]\\
    &+\lambda \langle F(z_k) - F(z_{k-1}), z_k - z_{k+1}\rangle.
    \end{aligned}
\end{equation}
We observe that the left-side of~\eqref{eq:add prox-ineqs} is non-negative as a consequence of~\eqref{eq:VI} and~\ref{ass:F}:
\begin{equation}\label{eq:lhs nonneg}
0\leq \langle F(z^*),z_k-z^*\rangle + g(z_k) - g(z^*)\leq \langle F(z_k),z_k-z^*\rangle + g(z_k) - g(z^*). 
\end{equation}
To estimate the final term in \eqref{eq:add prox-ineqs}, we use the Cauchy--Schwarz inequality, $L$-Lipschitz continuity of $F$, $\sigma$-strong convexity of $h$ and the inequality $\lambda\leq\frac{\sigma\varphi}{2L}$ to obtain
\begin{equation}\label{eq:lipschitz}
\begin{aligned}
    \lambda\langle F(z_k) - F(z_{k-1}), z_k - z_{k+1}\rangle 
    &\leq \lambda L\|z_k - z_{k-1}\|\|z_k - z_{k+1}\|\\
    &\leq \frac{\lambda L}{2}\left(\|z_k - z_{k-1}\|^2+\|z_k - z_{k+1}\|^2\right)\\
    &\leq \frac{\varphi}{2}\left(D_h(z_k, z_{k-1}) + D_h(z_{k+1}, z_k)\right).
\end{aligned}
\end{equation}
Combining~\eqref{eq:add prox-ineqs}, \eqref{eq:lhs nonneg} and~\eqref{eq:lipschitz} gives
 \begin{multline}\label{eq:combination}
         D_h(z^*, z_{k+1}) \leq D_h(z^*, \overline{z}_k) + (\varphi-1)D_h(z_{k+1}, \overline{z}_k)  - \varphi D_h(z_k, \overline{z}_k)\\
         - \frac{\varphi}{2}D_h(z_{k+1}, z_k) + \frac{\varphi}{2}D_h(z_k,z_{k-1}).
 \end{multline}
Now applying Proposition~\ref{prop:prop}\ref{item:identity} with $\nabla h(z_{k+1}) = \frac{\varphi\nabla h(\overline{z}_{k+1}) - \nabla h(\overline{z}_k)}{\varphi-1} = (\varphi+1)\nabla h(\overline{z}_{k+1}) - \varphi\nabla h(\overline{z}_k)$ and rearranging yields
\begin{multline}\label{eq:identity}
 (\varphi+1)D_h(z^*, \overline{z}_{k+1}) = D_h(z^*, z_{k+1})  + (\varphi+1)D_h(z_{k+1}, \overline{z}_{k+1}) \\ + \varphi\Big[D_h(z^*, \overline{z}_k) - D_h(z_{k+1}, \overline{z}_k)\Big].
 \end{multline}
Combining \eqref{eq:combination} and \eqref{eq:identity}, followed by collecting like-terms, gives
\begin{multline}\label{eq:central}
 (\varphi+1)D_h(z^*, \overline{z}_{k+1}) + \frac{\varphi}{2}D_h(z_{k+1}, z_k) - \varphi D_h(z_{k+1}, \overline{z}_{k+1}) \\
    \leq (\varphi+1)D_h(z^*, \overline{z}_k) + \frac{\varphi}{2}D_h(z_k, z_{k-1}) - \varphi D_h(z_k, \overline{z}_k) \\
    + D_h(z_{k+1}, \overline{z}_{k+1}) - D_h(z_{k+1}, \overline{z}_k).
\end{multline}
Since $\nabla h(\overline{z}_{k+1})=\frac{\varphi-1}{\varphi}\nabla h(z_{k+1})+\frac{1}{\varphi}\nabla h(\overline{z}_k)$, Proposition~\ref{prop:prop}\ref{item:identity} gives
\begin{multline}\label{eq:distance}
D_h(z_{k+1}, \overline{z}_{k+1}) = \frac{\varphi-1}{\varphi}\big[D_h(z_{k+1}, z_{k+1}) - D_h(\overline{z}_{k+1}, z_{k+1})\big] + \frac{1}{\varphi}\big[D_h(z_{k+1}, \overline{z}_k) - D_h(\overline{z}_{k+1}, \overline{z}_k)\big] \\\leq \frac{1}{\varphi}D_h(z_{k+1}, \overline{z}_k).
\end{multline}
Combining \eqref{eq:central} and \eqref{eq:distance} establishes the second inequality in~\eqref{eq:energy}. To show the first inequality in~\eqref{eq:energy}, we apply the three point identity (Proposition~\ref{prop:prop}\ref{item:3}) to see that
\begin{equation}
    \begin{aligned}\label{eq:3pt2}
       D_h(z^*,z_{k+1}) +& D_h(z_{k+1},\overline{z}_{k+1}) = D_h(z^*,\overline{z}_{k+1}) + \langle\nabla h(\overline{z}_{k+1}) - \nabla h(z_{k+1}), z^*-z_{k+1}\rangle\\
       &= D_h(z^*,\overline{z}_{k+1}) + \frac{1}{\varphi}\langle\nabla h(\overline{z}_k) - \nabla h(z_{k+1}), z^*-z_{k+1}\rangle.\\ 
          &\leq D_h(z^*, \overline{z}_{k+1}) + \frac{\lambda}{\varphi}\left(\langle F(z_k), z^* - z_{k+1}\rangle - g(z_{k+1}) + g(z^*)\right)\\
      &= D_h(z^*, \overline{z}_{k+1}) + \frac{\lambda}{\varphi}\langle F(z_k) - F(z_{k+1}), z^* - z_{k+1}\rangle \\ &\qquad +\frac{\lambda}{\varphi}(\langle F(z_{k+1}), z^* - z_{k+1}\rangle - g(z_{k+1}) + g(z^*)) \\
      &\leq D_h(z^*, \overline{z}_{k+1}) + \frac{\lambda}{\varphi}\langle F(z_k) - F(z_{k+1}), z^* - z_{k+1}\rangle.
\end{aligned}
\end{equation}
Using $L$-Lipschitz continuity of $F$ and $\sigma$-strong convexity of $h$ gives
\begin{equation}\label{eq:lipschitz2}
    \frac{\lambda}{\varphi}\langle F(z_k) - F(z_{k+1}), z^*- z_{k+1}\rangle \leq \frac{1}{2}D_h(z_{k+1}, z_k) + \frac{1}{2}D_h(z^*, z_{k+1}).
\end{equation}
On substituting~\eqref{eq:lipschitz2} back into~\eqref{eq:3pt2} and rearranging, we obtain
\begin{equation}\label{eq:3pt3}
\begin{aligned}
 D_h(z_{k+1}, \overline{z}_{k+1}) 
 &\leq D_h(z^*, \overline{z}_{k+1}) + \frac{1}{2}D_h(z_{k+1}, z_k)  -\frac{1}{2}D_h(z^*, z_{k+1})\\
 &\leq D_h(z^*, \overline{z}_{k+1}) + \frac{1}{2}D_h(z_{k+1}, z_k),
\end{aligned}
\end{equation}
and therefore
\begin{multline*}
    0 \leq D_h(z^*, \overline{z}_{k+1}) + \frac{1}{2}D_h(z_{k+1}, z_k) - D_h(z_{k+1}, \overline{z}_{k+1}) \\
    \implies 0\leq (\varphi+1)D_h(z^*, \overline{z}_{k+1}) + \frac{\varphi}{2}D_h(z_{k+1}, z_k) - \varphi D_h(z_{k+1}, \overline{z}_{k+1}),
\end{multline*}
which establishes the first inequality of~\eqref{eq:energy} and thus completes the proof.
\end{proof}

The following is our main result regarding convergence of the Bregman GRAAL with fixed step-size.
\begin{theorem}\label{thm:conv}
Suppose Assumptions~\ref{ass:g}-\ref{ass:S} hold and that $F$ is $L$-Lipschitz continuous on $\dom g\cap\intr\dom h$. Then the sequences $(z_k)$ and $(\overline{z}_k)$ generated by Algorithm~\ref{alg:B-GRAAL} converge to a point in $ S\cap\dom h$.
\end{theorem}
\begin{proof}
Let $z^*\in \Omega$ be arbitrary and let $(\eta_k)$ denote the sequence given by
$$\eta_k := (\varphi+1)D_h(z^*, \overline{z}_k) + \frac{\varphi}{2}D_h(z_k, z_{k-1}) - \varphi D_h(z_k, \overline{z}_k)\quad\forall k\in\mathbb{N}.$$
 Lemma~\ref{lem:dec} implies that $\lim_{k\to\infty}\eta_k$ exists and $D_h(z_{k+1}, \overline{z}_k) \to 0$ as $k\to\infty$. Referring to~\eqref{eq:distance}, it follows that $D_h(z_{k+1}, \overline{z}_{k+1}) \to 0$. Also, by applying Proposition~\ref{prop:prop}\ref{item:identity} with the identity $\nabla h(z_k) = (\varphi+1)\nabla h(\overline{z}_k) - \varphi \nabla h(\overline{z}_{k-1})$, we obtain
\begin{align*}
D_h(z_{k+1}, z_k) 
&= (\varphi+1)\big[D_h(z_{k+1}, \overline{z}_k)
- D_h(z_k,\overline{z}_k)\big]- \varphi\big[D_h(z_{k+1}, \overline{z}_{k-1}) - D_h(z_k, \overline{z}_{k-1})\big] \\
&\leq (\varphi+1)D_h(z_{k+1},\overline{z}_k)+\varphi D_h(z_k,\overline{z}_{k-1}) \to 0.
\end{align*}
Altogether, we have that
 $$ \lim_{k\to\infty}\eta_k = (\varphi+1)\lim_{k\to\infty}D_h(z^*, \overline{z}_k),$$
and, in particular, $\lim_{k\to\infty}D_h(z^*, \overline{z}_k)$ exists.
 
Next, using $\sigma$-strong convexity of $h$, we deduce that $z_{k+1}-\overline{z}_k\to0$, and that $(z_k)$ and $(\overline{z}_k)$ are bounded. Thus, let $\overline{z}\in\Hilbert$ be a cluster point of $(\overline{z}_k)$. Then there exists a subsequence $(\overline{z}_{k_j})$ such that $\overline{z}_{k_j}\to \overline{z}$ and $z_{k_j+1}\to \overline{z}$ as $j\to\infty$.
Now recalling~\eqref{eq:prox-ineq} gives
$$\lambda\left(\langle F(z_{k_j}),z_{{k_j}+1}-z\rangle + g(z_{{k_j}+1}) - g(z)\right) \leq \langle \nabla h(\overline{z}_{k_j}) - \nabla h(z_{{k_j}+1}), z_{{k_j}+1} - z\rangle \quad\forall z\in\Hilbert,$$
and taking the limit-infimum of both sides as $j\to\infty$ shows that $\overline{z}\in\Omega$. Since $z^*\in\Omega$ was chosen in Lemma~\ref{lem:dec} to be arbitrary, we can now set $z^*=\overline{z}$. It then follows that $\lim_{j\to\infty}D_h(z^*,\overline{z}_{k_j}) = 0$, and consequently, $\lim_{j\to\infty}\eta_{k_j} = 0$. Also note that for $n\geq k_j$, we have $\eta_n\leq\eta_{k_j}$ from Lemma \ref{lem:dec}, and therefore
$$(\varphi+1)\lim_{n\to\infty}D_h(z^*, \overline{z}_n) = \lim_{n\to\infty}\eta_n \leq \lim_{j\to\infty}\eta_{k_j} = 0,$$
and therefore $\overline{z}_k\to z^*$ from strong convexity. The fact that $z_k\to z^*$ follows since $z_k-\overline{z}_k\to 0$.
\end{proof}

\begin{remark}
In the special case where $h=\|\cdot\|^2$, Algorithm~\ref{alg:B-GRAAL} recovers the Euclidean GRAAL with fixed step-size from \cite[Section~2]{malitsky2020golden} and the conclusions of Theorem~\ref{thm:conv} recover \cite[Theorem~1]{malitsky2020golden}. Despite this, the proof provided here is new and not the same as the one in \cite[Theorem~1]{malitsky2020golden} even when specialised to the Euclidean case. Indeed, \cite[Theorem~1]{malitsky2020golden} proceeds by establishing the inequality 
\begin{equation}\label{eq:mal}
    (\varphi+1)\|\overline{z}_{k+1}-z^*\|^2 + \frac{\varphi}{2}\|z_{k+1}-z_k\|^2 \\
    \leq (\varphi+1)\|\overline{z}_{k}-z^*\|^2 + \frac{\varphi}{2}\|z_{k}-z_{k-1}\|^2 - \varphi\|z_k-\overline{z}_k\|^2,
\end{equation}
which is different to Lemma~\ref{lem:dec}. Interestingly, \eqref{eq:mal} can be deduced from \eqref{eq:central} in Lemma~\ref{lem:dec} by using the equality
\begin{equation*}
\|z_{k+1}-\overline{z}_{k+1}\|^2 = \frac{1}{\varphi^2}\|z_{k+1} - \overline{z}_k\|^2,
\end{equation*}
which applies in the Euclidean case, in place of \eqref{eq:distance} followed by the identity $\varphi^2=\varphi+1$. Note also that the inequality \eqref{eq:distance} is already weaker than the inequality $\|z_{k+1}-\overline{z}_{k+1}\|^2 \leq \frac{1}{\varphi^2}\|z_{k+1} - \overline{z}_k\|^2$. 
\end{remark}

\section{The Adaptive Bregman Golden Ratio Algorithm}\label{sec:bagraal}
In this section, we present an adaptive modification to Algorithm~\ref{alg:B-GRAAL} and analyse its convergence. The method is presented in Algorithm~\ref{alg:B-aGRAAL}. As with the Euclidean aGRAAL, our Bregman adaptive modification has a fully explicit step-size rule. It is presented in Algorithm~\ref{alg:B-aGRAAL}.

\begin{algorithm}[!htb]
\caption{The Bregman adaptive Golden RAtio ALgorithm (B-aGRAAL)}\label{alg:B-aGRAAL}
    \SetKwInput{Init}{Initialisation}
    \KwIn{Initial points $z_0,\overline{z}_0\in\dom g$, initial step-size $\lambda_0>0$, $\lambda_{\max}\gg 0, \phi\in(1,\varphi]$}
    \Init{Set $z_1=\overline{z}_0$, $\theta_0=1$, choose $\rho\in\left[1,\frac{1}{\phi}+\frac{1}{\phi^2}\right]$}
    \For{$k=1,2,\dots$}{
    Calculate the step-size:
    \begin{equation}\label{eq:lambda}
    \lambda_k = \min\left\{\rho\lambda_{k-1}, \frac{\sigma\phi\theta_{k-1}}{4\lambda_{k-1}}\frac{\|z_k - z_{k-1}\|^2}{\|F(z_k) - F(z_{k-1})\|^2}, \lambda_{\max}\right\}.
    \end{equation}\\
    Compute the next iterates:
    \begin{align}
        \overline{z}_k &= (\nabla h)^{-1}\left( \frac{(\phi - 1)\nabla h(z_k) + \nabla h(\overline{z}_{k-1})}{\phi}\right)\label{eq:bar}\\
        z_{k+1} &= \argmin_{z\in\Hilbert}\left\{\langle F(z_k), z-z_k\rangle + g(z) + \frac{1}{\lambda_k} D_h(z,\overline{z}_k)\right\}\label{eq:z}\\
    \end{align}\\
    Update: $\theta_k = \frac{\lambda_k\phi}{\lambda_{k-1}}.$ 
    }
\end{algorithm}

Observe that the step-size sequence $(\lambda_k)$ in Algorithm~\ref{alg:B-aGRAAL} approximates the inverse of a local Lipschitz constant in the following sense:
\begin{multline}\label{eq:inv loc lip}
    \lambda_k \leq \frac{\sigma\phi\theta_{k-1}}{4\lambda_{k-1}}\frac{\|z_k - z_{k-1}\|^2}{\|F(z_k) - F(z_{k-1})\|^2} \leq \frac{\sigma\theta_k\theta_{k-1}}{4\lambda_k}\frac{\|z_k - z_{k-1}\|^2}{\|F(z_k) - F(z_{k-1})\|^2} \\
    \implies \lambda_k\|F(z_k) - F(z_{k-1})\| \leq \frac{\sqrt{\sigma\theta_k\theta_{k-1}}}{2}\|z_k - z_{k-1}\|.
\end{multline}

Before giving our main convergence result for Algorithm~\ref{alg:B-aGRAAL}, we require some preparatory lemmas. The first two are concerned with well-definedness of the algorithm and boundedness of the step-size sequence.
\begin{lemma}\label{lem:ad contained}
Suppose Assumptions~\ref{ass:g}-\ref{ass:h} hold. Then the sequences $(\overline{z}_k)$ and $(z_k)$ generated by Algorithm~\ref{alg:B-aGRAAL} are well-defined. Moreover, $(\overline{z}_k)\subseteq\intr\dom h$ and  $(z_k)\subseteq\intr\dom h\cap \dom g$.
\end{lemma}
\begin{proof}
Follows by an analogous argument to that of Lemma~\ref{lem:contained} but with $\lambda$ replaced by $\lambda_k$ and $\varphi$ replaced by $\phi$.
\end{proof}
\begin{lemma}\label{lem:bounded}
If $(z_k)$ generated by Algorithm~\ref{alg:B-aGRAAL} is bounded and $F$ is locally Lipschitz, then both $(\lambda_k)$ and $(\theta_k)$ are bounded and separated from $0$. In fact, for any $L>0$ satisfying $\|F(z_k)-F(z_{k-1})\|\leq L\|z_k-z_{k-1}\|$ for all $k\in\mathbb{N}$, we have
$$\lambda_k\geq \frac{\sigma\phi^2}{4L^2\lambda_{\max}}\text{~~and~~} \theta_k\geq\frac{\sigma\phi^3}{4L^2\lambda_{\max}^2}\quad\forall k\in\mathbb{N}. $$
\end{lemma}
\begin{proof}
First we note that $\lambda_k\leq\lambda_{\max}$ by definition, and that $\theta_k\leq\rho\phi\leq 1+\frac{1}{\phi}$.
Since $(z_k)$ is bounded and $F$ is locally Lipschitz continuous, there exists $L>0$ such that $\|F(z_k)-F(z_{k-1})\|\leq L\|z_k-z_{k-1}\|$ for all $k\in\mathbb{N}$. Then an argument analogous to that of~\cite[Lemma 2]{malitsky2020adaptive} shows that $\lambda_k\geq \frac{\sigma\phi^2}{4L^2\lambda_{\max}}$, and consequently, $\theta_k\geq\frac{\sigma\phi^3}{4L^2\lambda_{\max}^2}$.
\end{proof}
Our next result establishes an inequality similar, but not completely analogous, to its fixed step-size counterpart in Lemma~\ref{lem:dec}.
\begin{lemma}\label{lem:adaptive energy}
Suppose Assumptions~\ref{ass:g}-\ref{ass:S} hold. Let $z^*\in \Omega$ be arbitrary. Then the sequences $(z_k)$, $(\overline{z}_k)$ generated by Algorithm~\ref{alg:B-aGRAAL} satisfy.
\begin{multline}\label{eq:ad energy}
 \frac{\phi}{\phi-1}D_h(z^*, \overline{z}_{k+1}) + \frac{\theta_k}{2}D_h(z_{k+1}, z_k) - D_h(z_{k+1}, \overline{z}_{k+1}) \\
    \leq \frac{\phi}{\phi-1}D_h(z^*, \overline{z}_k) + \frac{\theta_{k-1}}{2}D_h(z_k, z_{k-1}) 
    - \theta_k D_h(z_k, \overline{z}_k) \\+\left(\theta_k-1-\frac{1}{\phi}\right)D_h(z_{k+1}, \overline{z}_k).
\end{multline}
\end{lemma}

\begin{proof}
We proceed in a similar fashion as in the proof to Lemma~\ref{lem:dec}. By first applying Proposition~\ref{prop:prox}\ref{item:prox thm} with $f(z) := \lambda_k(\langle F(z_k), z-z_k\rangle + g(z))$, $u:=z\in\dom h\cap\dom g$ arbitrary, $x:=z_{k+1}$ and $y_0:=\overline{z}_k$, followed by the three-point identity (Proposition~\ref{prop:prop}\ref{item:3}) we obtain
\begin{equation}\label{eq:ad prox-ineq}
\begin{aligned}
\lambda_k\left(\langle F(z_k),z_{k+1}-z\rangle + g(z_{k+1}) - g(z)\right) &\leq \langle \nabla h(\overline{z}_k) - \nabla h(z_{k+1}), z_{k+1} - z\rangle\\
&= D_h(z, \overline{z}_k) - D_h(z, z_{k+1}) - D_h(z_{k+1}, \overline{z}_k).
\end{aligned}
\end{equation}
Shifting the index~\eqref{eq:ad prox-ineq} (by setting $k\equiv k-1$), setting $z:=z_{k+1}$, and using the fact that $\phi\nabla h(\overline{z}_k) = (\phi-1)\nabla h(z_k) + \nabla h(\overline{z}_{k-1})$ followed by the three point identity (Proposition~\ref{prop:prop}\ref{item:3}) gives
\begin{equation}\label{eq:ad shift prox-ineq}
\begin{aligned}
    &\lambda_{k-1}\left(\langle F(z_{k-1}), z_k - z_{k+1}\rangle + g(z_k) - g(z_{k+1})\right) \\
    &\quad \leq \langle \nabla h(\overline{z}_{k-1}) - \nabla h(z_k), z_k - z_{k+1}\rangle\\
   &\quad= \phi\langle\nabla h(\overline{z}_k) - \nabla h(z_k), z_k - z_{k+1}\rangle \\
   &\quad = \phi\left[D_h(z_{k+1},\bar{z}_k)-D_h(z_{k+1},z_k)-D_h(z_k,\bar{z}_k)\right].
   \end{aligned}
   \end{equation}
Now multiplying both sides of~\eqref{eq:ad shift prox-ineq} by $\tfrac{\lambda_k}{\lambda_{k-1}}$ then gives
\begin{multline}\label{eq:ad shift prox-ineq2}
 \lambda_k\left(\langle F(z_{k-1}), z_k - z_{k+1}\rangle + g(z_k) - g(z_{k+1})\right) \\ \leq\theta_k\left[D_h(z_{k+1}, \overline{z}_k) - D_h(z_{k+1}, z_k) - D_h(z_k, \overline{z}_k)\right].
\end{multline}
Let $z^*\in S\cap\dom h$. Setting $z=z^*$ in~\eqref{eq:ad prox-ineq}, summing with~\eqref{eq:ad shift prox-ineq}  and rearranging yields
\begin{equation}\label{eq:ad add prox-ineqs}
    \begin{aligned}
    \lambda_k\left(\langle F(z_k), z_k - z^*\rangle + g(z_k) - g(z^*)\right) \leq & D_h(z^*,\overline{z}_k) - D_h(z^*,z_{k+1}) - D_h(z_{k+1},\overline{z}_k)\\
    &+\theta_k\Big[D_h(z_{k+1},\overline{z}_k) - D_h(z_{k+1},z_k) - D_h(z_k,\overline{z}_k)\Big]\\
    &+\lambda_k\langle F(z_k) - F(z_{k-1}), z_k - z_{k+1}\rangle.
    \end{aligned}
\end{equation}
We observe that the left-side of~\eqref{eq:ad add prox-ineqs} is non-negative as a consequence of~\eqref{eq:VI} and~\ref{ass:F}:
\begin{equation}\label{eq:ad lhs nonneg}
0\leq \langle F(z^*),z_k-z^*\rangle + g(z_k) - g(z^*)\leq \langle F(z_k),z_k-z^*\rangle + g(z_k) - g(z^*).
\end{equation}
To estimate the final term of~\eqref{eq:ad add prox-ineqs} we use the Cauchy--Schwarz inequality, the local Lipschitz estimate~\eqref{eq:inv loc lip}, and $\sigma$-strong convexity of $h$ to obtain
\begin{equation}\label{eq:local lipschitz}
\begin{aligned}
    \lambda_k\langle F(z_k) - F(z_{k-1}), z_k - z_{k+1}\rangle 
    &\leq \frac{\sqrt{\sigma\theta_k\theta_{k-1}}}{2}\|z_k - z_{k-1}\|\|z_k - z_{k+1}\|\\
    &\leq \frac{\sigma\theta_{k-1}}{4}\|z_k - z_{k-1}\|^2+\frac{\sigma\theta_k}{4}\|z_k - z_{k+1}\|^2\\
    &\leq \frac{\theta_{k-1}}{2}D_h(z_k, z_{k-1}) + \frac{\theta_k}{2}D_h(z_{k+1}, z_k).
\end{aligned}
\end{equation}
Combining~\eqref{eq:ad add prox-ineqs}, \eqref{eq:ad lhs nonneg} and~\eqref{eq:local lipschitz} gives
 \begin{multline}\label{eq:ad combination}
         D_h(z^*, z_{k+1}) \leq D_h(z^*, \overline{z}_k) + (\theta_k-1)D_h(z_{k+1}, \overline{z}_k)  - \theta_k D_h(z_k, \overline{z}_k)\\
         - \frac{\theta_k}{2}D_h(z_{k+1}, z_k) + \frac{\theta_{k-1}}{2}D_h(z_k,z_{k-1}).
 \end{multline}
Now applying Proposition~\ref{prop:prop}\ref{item:identity} with $\nabla h(z_{k+1}) = \frac{\phi\nabla h(\overline{z}_{k+1}) - \nabla h(\overline{z}_k)}{\phi-1}$ and rearranging yields
\begin{multline}\label{eq:ad identity}
 \frac{\phi}{\phi-1}D_h(z^*, \overline{z}_{k+1}) = D_h(z^*, z_{k+1})  + \frac{\phi}{\phi-1}D_h(z_{k+1}, \overline{z}_{k+1}) \\ + \frac{1}{\phi-1}\Big[D_h(z^*, \overline{z}_k) - D_h(z_{k+1}, \overline{z}_k)\Big].
 \end{multline}
Combining \eqref{eq:ad combination} and \eqref{eq:ad identity}, followed by collecting like-terms and rearranging, gives
\begin{multline}\label{eq:ad central}
 \frac{\phi}{\phi-1}D_h(z^*, \overline{z}_{k+1}) + \frac{\theta_k}{2}D_h(z_{k+1}, z_k) -  D_h(z_{k+1}, \overline{z}_{k+1}) \\
    \leq \frac{\phi}{\phi-1}D_h(z^*, \overline{z}_k) + \frac{\theta_{k-1}}{2}D_h(z_k, z_{k-1}) 
    - \theta_k D_h(z_k, \overline{z}_k) \\+\left(\theta_k - \frac{\phi}{\phi-1}\right)D_h(z_{k+1}, \overline{z}_k) + \frac{1}{\phi-1}D_h(z_{k+1}, \overline{z}_{k+1}).
\end{multline}
Next we apply Proposition~\ref{prop:prop}\ref{item:identity} once again to see that
\begin{multline}\label{eq:ad distance}
    D_h(z_{k+1},\overline{z}_{k+1})=\frac{\phi-1}{\phi}\left[D_h(z_{k+1},z_{k+1}) - D_h(\overline{z}_{k+1},z_{k+1})\right] + \frac{1}{\phi}\left[D_h(z_{k+1},\overline{z}_k) - D_h(\overline{z}_{k+1},\overline{z}_k)\right]\\
    \leq\frac{1}{\phi}D_h(z_{k+1},\overline{z}_k).
\end{multline}
The final line of~\eqref{eq:ad central} can therefore be estimated as
\begin{equation}\label{eq:last adp}\begin{aligned}
\left(\theta_k - \frac{\phi}{\phi-1}\right)D_h(z_{k+1}, \overline{z}_k) + \frac{1}{\phi-1}D_h(z_{k+1}, \overline{z}_{k+1})
&\leq \left(\theta_k - \frac{\phi}{\phi-1} + \frac{1/\phi}{\phi-1}\right)D_h(z_{k+1}, \overline{z}_k) \\
&= \left(\theta_k -1-\frac{1}{\phi}\right)D_h(z_{k+1}, \overline{z}_k).
\end{aligned}\end{equation}
Substituting \eqref{eq:last adp} into \eqref{eq:ad central} gives \eqref{eq:ad energy}, which completes the proof.
\end{proof}

The following is our main result regarding convergence of the Bregman adaptive GRAAL.
\begin{theorem}\label{th:B-aGRAAL}
 Suppose Assumptions~\ref{ass:g}-\ref{ass:S} hold, and $F$ is locally Lipschitz continuous on the bounded set $\dom g\cap\intr\dom h$. Choose $\phi\in\left(1,\varphi\right)$ and $\rho\in\left[1,\frac{1}{\phi}+\frac{1}{\phi^2}\right)$.
 Then the sequences $(z_k)$ and $(\overline{z}_k)$ generated by Algorithm~\ref{alg:B-aGRAAL} converge to a point in $\Omega$ whenever $\lambda_{\rm max}>0$ is sufficiently small.
 \end{theorem}

\begin{proof}
Since $\dom g\cap\intr\dom h$ is bounded and $F$ is locally Lipschitz, there exists $L>0$ such that $F$ is $L$-Lipschitz continuous on $\dom g\cap\intr\dom h$. Suppose $\lambda_{\rm max}$ is sufficient small so that $\lambda_{\max}\leq\frac{\phi\sqrt{\phi\sigma}}{2L}$ holds. Then, applying Lemma~\ref{lem:bounded}  gives $\theta_k\geq\frac{\sigma\phi^3}{4L^2\lambda_{\max}^2}\geq 1$ for all $k\in\mathbb{N}$. Also, since $\theta_k \leq \rho\phi < 1+\frac{1}{\phi}$, there exists an $\epsilon>0$ such that $\theta_k-1-\frac{1}{\phi} \leq -\epsilon$ for all $k\in\mathbb{N}$.

Now let $z^*\in\Omega$ be arbitrary and denote by $(\eta_k)$ the sequence
$$\eta_k := \frac{\phi}{\phi-1}D_h(z^*, \overline{z}_k) +\frac{\theta_{k-1}}{2}D_h(z_k, z_{k-1}) - D_h(z_k, \overline{z}_k).$$
Applying Lemma~\ref{lem:adaptive energy} yields
\begin{equation}\label{eq:new energy}
 \eta_{k+1}  \leq \eta_k -\epsilon D_h(z_{k+1},\overline{z}_k).
\end{equation}
Next, we should show that $(\eta_k)$ is bounded below. To this end, using the three point identity (Proposition \ref{prop:prop}\ref{item:3}) followed by \eqref{eq:ad prox-ineq} gives
\begin{equation}\label{eq:last term 1}
\begin{aligned}
      D_h(z^*,z_{k+1}) &+ D_h(z_{k+1},\overline{z}_{k+1}) = D_h(z^*,\overline{z}_{k+1}) + \langle\nabla h(\overline{z}_{k+1}) - \nabla h(z_{k+1}), z^*-z_{k+1}\rangle\\ 
      &= D_h(z^*,\overline{z}_{k+1}) + \frac{1}{\phi}\langle\nabla h(\overline{z}_k) - \nabla h(z_{k+1}), z^*-z_{k+1}\rangle\\ 
      &\leq D_h(z^*, \overline{z}_{k+1}) + \frac{\lambda_k}{\phi}\left(\langle F(z_k), z^* - z_{k+1}\rangle - g(z_{k+1}) + g(z^*)\right)\\
      &= D_h(z^*, \overline{z}_{k+1}) + \frac{\lambda_k}{\phi}\langle F(z_k) - F(z_{k+1}), z^* - z_{k+1}\rangle \\ &\quad +\frac{\lambda_k}{\phi}(\langle F(z_{k+1}), z^* - z_{k+1}\rangle - g(z_{k+1}) + g(z^*)).
\end{aligned}
\end{equation}
Now, as with~\eqref{eq:add prox-ineqs}, the final term in \eqref{eq:last term 1} is non-negative. Since $(\lambda_k)$ and $(\theta_k)$ are bounded and separated from $0$ by Lemma~\ref{lem:bounded}, there exists a constant $M>0$ such that
\begin{equation}\label{eq:bd above}\begin{aligned}
\lambda_k\langle F(z_k) - F(z_{k+1}), z^* - z_{k+1}\rangle
&\leq \frac{\lambda_k}{\lambda_{k+1}}\lambda_{k+1}\|F(z_k) - F(z_{k+1})\|\|z^*-z_{k+1}\| \\
&\leq \frac{\lambda_k}{\lambda_{k+1}}\frac{\sigma\sqrt{\theta_{k+1}\theta_k}}{2}\|z_k-z_{k+1}\|\|z^*-z_{k+1}\|\leq M.
\end{aligned}\end{equation}
Combing \eqref{eq:last term 1} and \eqref{eq:bd above}, noting that $\phi<\frac{\phi}{\phi-1}$, then gives
\begin{equation}\label{eq:bd below}
    D_h(z_{k+1},\overline{z}_{k+1}) \leq D_h(z^*, \overline{z}_{k+1}) + \frac{M}{\phi}\leq\phi D_h(z^*,\overline{z}_{k+1}) + M
    \implies \eta_{k+1} \geq -M,
\end{equation}
which establishes that $(\eta_k)$ is bounded below.

Next, by telescoping the inequality~\eqref{eq:new energy}, deduce that $(\eta_k)$ is bounded and $D_h(z_{k+1},\overline{z}_k)\to 0$ as $k\to\infty$. Referring to~\eqref{eq:ad distance}, it follows that $D_h(z_{k+1},\overline{z}_{k+1})\to 0$. Also, by applying Proposition~\ref{prop:prop}\ref{item:identity} with the identity $\nabla h(z_k) = \frac{\phi\nabla h(\overline{z}_k) - \nabla h(\overline{z}_{k-1})}{\phi-1}$, we obtain
\begin{align*}
    D_h(z_{k+1},z_k) &= \frac{\phi}{\phi-1}\left[D_h(z_{k+1},\overline{z}_k) - D_h(z_k,\overline{z}_k)\right] - \frac{1}{\phi-1}\left[D_h(z_{k+1},\overline{z}_{k-1}) - D_h(z_k,\overline{z}_{k-1})\right]\\
                     &\leq \frac{\phi}{\phi-1}D_h(z_{k+1},\overline{z}_k) + \frac{1}{\phi-1}D_h(z_k,\overline{z}_{k-1}) \to 0.
\end{align*}
Altogether, noting that $(\theta_k)$ is bounded, we deduce that
$$\lim_{k\to\infty}\eta_k = \frac{\phi}{\phi-1}\lim_{k\to\infty}D_h(z^*,\overline{z}_k),$$
and so, in particular, $\lim_{k\to\infty}D_h(z^*,\overline{z}_k)$ exists.

Next, using $\sigma$-strong convexity of $h$, we deduce that $z_{k+1}-\overline{z}_k\to0$. Let $\overline{z}\in\Hilbert$ be a cluster point of the bounded sequence $(\overline{z}_k)$. Then there exists a subsequence $(\overline{z}_{k_j})$ such that $\overline{z}_{k_j}\to \overline{z}$ and $z_{k_j+1}\to \overline{z}$ as $j\to\infty$.
Now recalling~\eqref{eq:ad prox-ineq} gives
$$\lambda_k\left(\langle F(z_{k_j}),z_{{k_j}+1}-z\rangle + g(z_{{k_j}+1}) - g(z)\right) \leq \langle \nabla h(\overline{z}_{k_j}) - \nabla h(z_{{k_j}+1}), z_{{k_j}+1} - z\rangle \quad\forall z\in\Hilbert,$$
and taking the limit-infimum of both sides as $j\to\infty$ shows that $\overline{z}\in\Omega$. Since $z^*\in\Omega$ was chosen in Lemma~\ref{lem:adaptive energy} to be arbitrary, we can now set $z^*=\overline{z}$. It then follows that $\lim_{j\to\infty}D_h(z^*,\overline{z}_{k_j}) = 0$, and consequently, $\lim_{j\to\infty}\eta_{k_j} = 0$. Also note that for $n\geq k_j$, we have $\eta_n\leq\eta_{k_j}$ from Lemma \ref{lem:dec}, and therefore
$$\frac{\phi}{\phi-1}\lim_{n\to\infty}D_h(z^*, \overline{z}_n) = \lim_{n\to\infty}\eta_n \leq \lim_{j\to\infty}\eta_{k_j} = 0,$$
and $\overline{z}_k\to z^*$ from strong convexity. The fact that $z_k\to z^*$ follows since $z_k-\overline{z}_k\to 0$.
\end{proof}

\begin{remark}
The energy~\eqref{eq:new energy} used in the proof Theorem~\ref{th:B-aGRAAL} is not completely analogous to the one used in the fixed step-size case given in Lemma~\ref{lem:dec}. Notably, the coefficient of the final term of $\eta_k$ in \eqref{eq:new energy} has coefficient $-1$ whereas the corresponding term in Lemma~\ref{lem:dec} has coefficient $-\varphi$. Although it is unlikely to be useful in practice, another interesting feature of the proof it that it requires the maximum step-size to satisfy the upper bound $$\lambda_{\rm max}\leq \frac{\phi\sqrt{\phi\sigma}}{2L} = \sqrt{\frac{\phi}{\sigma}}\frac{\sigma\phi}{2L}. $$
Note that, when $\phi>\sigma$, this upper bound is looser than the upper-bound of $\lambda<\frac{\sigma\phi}{2L}$ required for B-GRAAL (Algorithm~\ref{alg:B-GRAAL}). Thus, it is possible for maximum step-size of B-aGRAAL to be larger than the step-size required for B-GRAAL. This is the case, for instance, if $D_h$ is the KL divergence on the simplex in which case $\sigma=1<\phi$.

Also, the proof of Theorem~\ref{th:B-aGRAAL} required $L$ be a Lipschitz constant for $F$. However, thanks to Lemma~\ref{lem:bounded}, it would have sufficed for $L$ to satisfy the weaker Lipschitz-like inequality $\|F(z_k)-F(z_{k-1})\|\leq L\|z_k-z_{k-1}\|$ for all $k\in\mathbb{N}$. In turn, this would allow for a greater upper bound for $\lambda_{\rm max}$ which is inversely related to $L$.
\end{remark}

\section{Numerical Experiments}\label{sec:exp}
In this section, we present some experimental results for Algorithms~\ref{alg:B-GRAAL} and \ref{alg:B-aGRAAL}, to compare the respective original Euclidean methods against our new methods with respect to various choices of Bregman distances. All experiments are run in Python 3 on a Windows 10 machine with 8GB memory and an Intel(R) Core(TM) i7-10510U CPU @ 1.80GHz processor.

We consider three different problems, all of which can formulated as the variation inequality~\eqref{eq:VI2} for different choices of the function $g$ and the operator $F$. Note that the solutions of the variational inequality~\eqref{eq:VI2} can also be characterised as the monotone inclusion
$$\text{find~}z^*\in\Hilbert\text{~such that~}0\in F(z^*) + \partial g(z^*).$$
Thus, by noting that \eqref{eq:z} implies $0 \in F(z_k) + \partial g(z_{k+1}) + \frac{1}{\lambda_k}\left(\nabla h(z_{k+1}) - \nabla h(\overline{z}_k)\right)$, we monitor the quantity $J_k$ given by
\begin{equation}\label{eq:Jk}
J_k := \frac{1}{\lambda_k}\left(\nabla h(\overline{z}_k) - \nabla h(z_{k+1})\right) + F(z_{k+1}) - F(z_k) \in F(z_{k+1}) + \partial g(z_{k+1})
\end{equation}
as a natural residual for Algorithm~\ref{alg:B-aGRAAL}. The analogous expression for Algorithm~\ref{alg:B-GRAAL} is given by replacing $\lambda_k$ in \eqref{eq:Jk} with $\lambda$.

In all experiments, we run each algorithm on for the same (fixed) number of iterations on $10$ random instances of the chosen problem. The figures in each section show the decrease in residual over time, with each faint line representing one instance and the bold line showing the mean behaviour. We use the parameters $\phi=1.5,\lambda_{\max}=10^6$ throughout, and initial step-size and iterates as described in each section.

\subsection{Matrix Games}\label{subsec:matrix}
To test the fixed step-size B-GRAAL (Algorithm~\ref{alg:B-GRAAL}), we first consider the following matrix game between two players
\begin{equation}\label{eq:matrix}
\min_{x\in\Delta^n}\max_{y\in\Delta^n}\langle Mx,y\rangle,
\end{equation}
where $\Delta^n := \{x\in\mathbb{R}^n_+\colon \sum_{i=1}^n x_i = 1\}$ denotes the unit simplex and $M\in\mathbb{R}^{n\times n}$ a given matrix. This problem is of the form specified in \eqref{eq:saddle} and so can be formulated as the variational inequality~\eqref{eq:VI}.

In particular, we consider the specific problem in the form of \eqref{eq:matrix} of placing a server on a network $G = (V,E)$ with $n$ vertices in a way that minimises its response time. In this problem, a request will originate at some vertex $v_j\in V$, which is not known ahead of time, and the objective is to place the server at a vertex $v_i\in V$ such the response time, as measured by the graphical distance $d(v_i,v_j)$, is minimised. We consider the case where the request location $v_j\in V$ is a decision made by an adversary. The decision variable $x\in\Delta^n$ (resp.\ $y\in\Delta^n$) models mixed strategies for placement of the server (resp.\ request origin). In other words, $x_t$ (resp.\ $y_t$) is the probability of the server (resp.\ request origin) being located at the node $v_t$ for $t=1,\dots,n$. The matrix $M$ is taken to be the distance matrix of the graph $G$, that is $M_{ij} = d(v_i,v_j)$ for all vertices $v_i,v_j\in V$. In this way, the objective function in~\eqref{eq:matrix} measures the expected response time, which we would like to minimise while our adversary seeks to maximise it.

We compare Algorithm~\ref{alg:B-GRAAL} with the squared norm $h(z) = \frac{1}{2}\|z\|^2$, which generates the squared Euclidean distance $D_h(u,v) = \frac{1}{2}\|u-v\|^2$, and the negative entropy $h(z) = \sum_{i=1}^{2n} z_i\log z_i$, which generates the KL divergence $D_h(u,v) = \sum_{i=1}^{2n}u_i\log\frac{u_i}{v_i} + v_i - u_i$, where $z=(x,y)\in\mathbb{R}^{2n}$ is the concatenation of the two vectors $x$ and $y$. Both of these choices for $h$ are $1$-strongly convex on $\Delta^n$. A potential advantage of the the KL projection onto the simplex is that it has the simple closed-formed expression $x\mapsto\frac{x}{\|x\|_1}$. whereas the Euclidean projection has no closed form and takes $O(n\log n)$ time (see, for instance \cite{wang2012projection, chen2011projection, dai2022distributed}).  We run two experiments with $n=500$ and $n=1000$, and the results are shown respectively in Figures~\ref{fig:matrix500} and~\ref{fig:matrix1000}. Initial points are chosen as $x_0=y_0=\left(\frac{1}{n},\dots,\frac{1}{n}\right)\in\Delta^n$, and $z_0=(x_0,y_0)$, then $\overline{z}_0$ as a random perturbation of $z_0$. The Lipschitz constant of $F$ is computed as $L=\|M\|_2$, and the algorithm step-size is taken as $\lambda=\frac{\varphi}{2L}$. Under these conditions, convergence to a solution is guranteeed by Theorem~\ref{thm:conv}.

The residual, given by $\min_{i\leq k}\|J_i\|^2$, is shown in Figures~\ref{sfig:500} and \ref{sfig:1000}. The time per iteration is also shown in Figures~\ref{sfig:500time} and \ref{sfig:1000time}. Despite the KL projection being faster than the Euclidean projection, overall, the Euclidean method performed better in this instance.

We now move onto experiments for the adaptive Algorithm~\ref{alg:B-aGRAAL}.

\begin{figure*}[t!]
    \centering
    \begin{subfigure}[t]{0.5\textwidth}
        \centering
        \includegraphics[width=\textwidth]{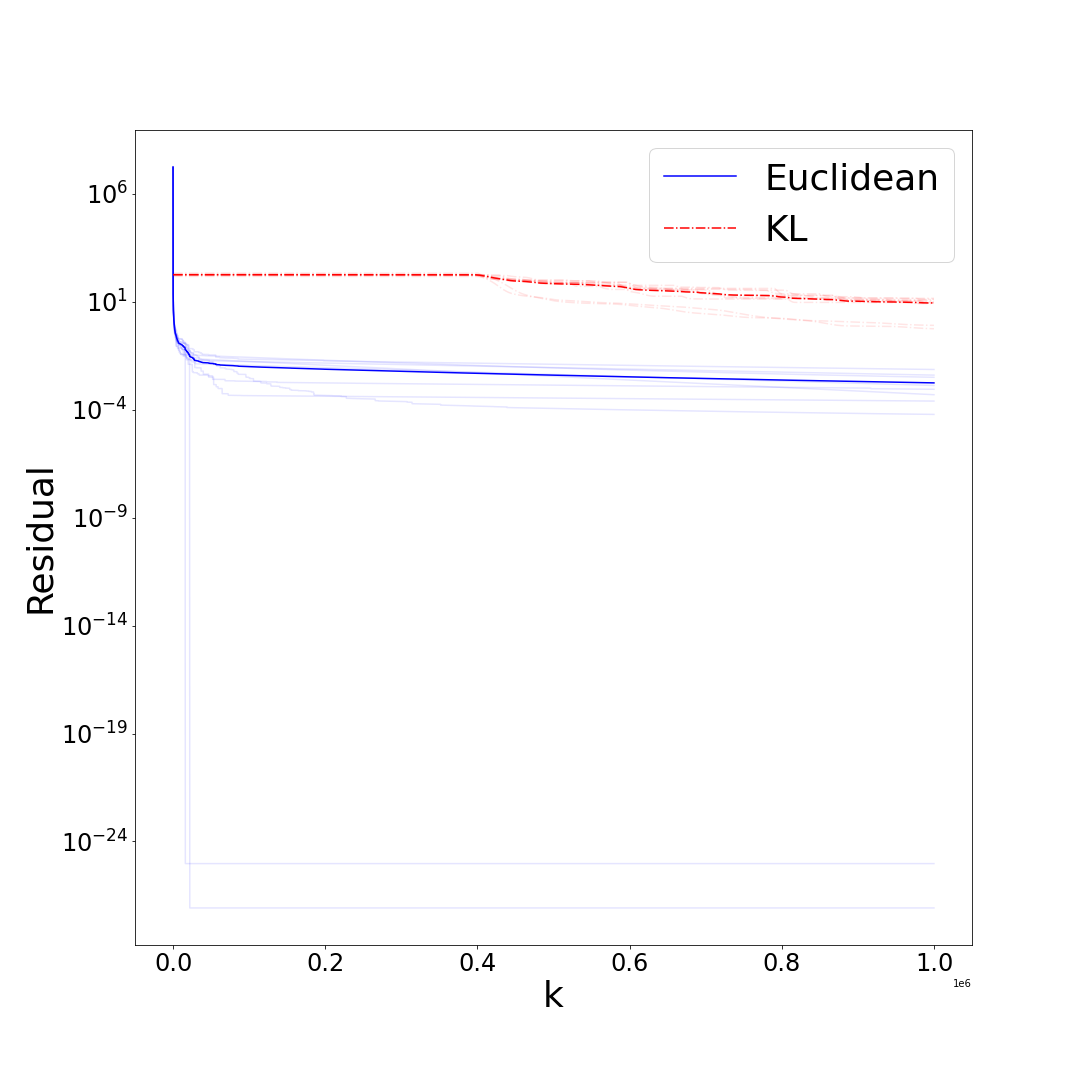}
        \caption{Residual over iterations}
        \label{sfig:500}
    \end{subfigure}%
    ~ 
    \begin{subfigure}[t]{0.5\textwidth}
        \centering
        \includegraphics[width=\textwidth]{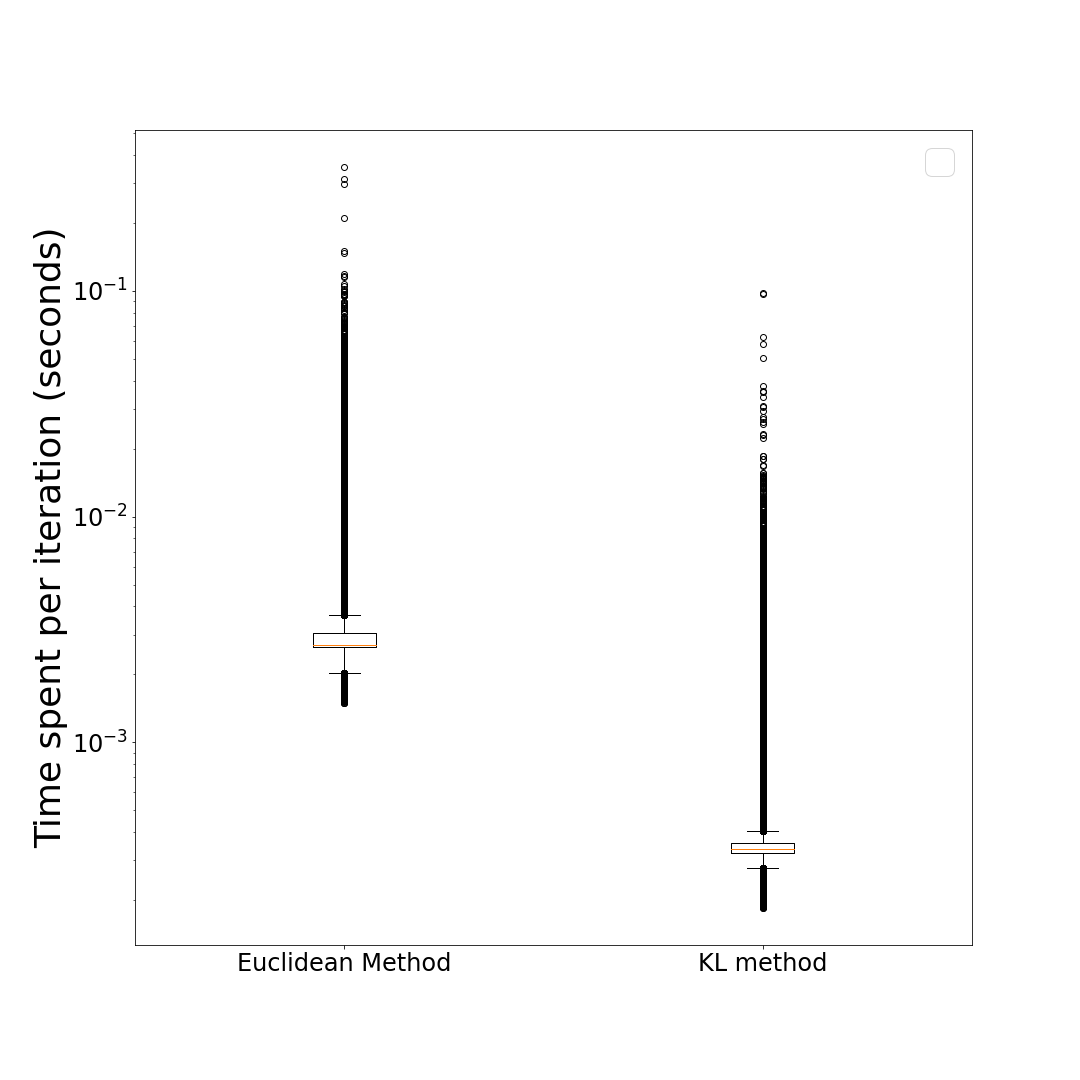}
        \caption{Run time per iteration}
        \label{sfig:500time}
    \end{subfigure}
    \caption{Matrix game results for $n=500$}
    \label{fig:matrix500}
\end{figure*}
\begin{figure*}[t!]
    \centering
    \begin{subfigure}[t]{0.5\textwidth}
        \centering
        \includegraphics[width=\textwidth]{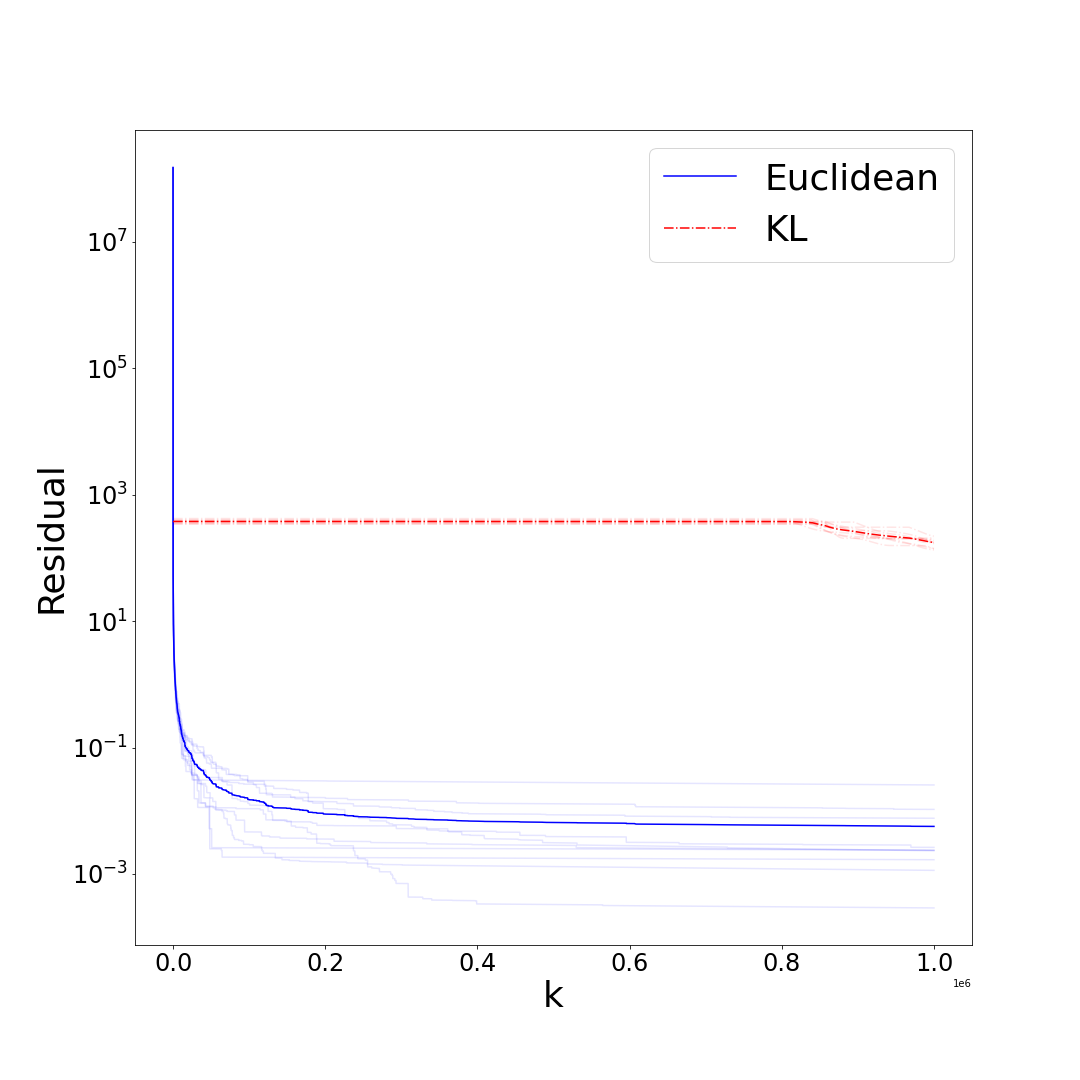}
        \caption{Residual over iterations}
        \label{sfig:1000}
    \end{subfigure}%
    ~ 
    \begin{subfigure}[t]{0.5\textwidth}
        \centering
        \includegraphics[width=\textwidth]{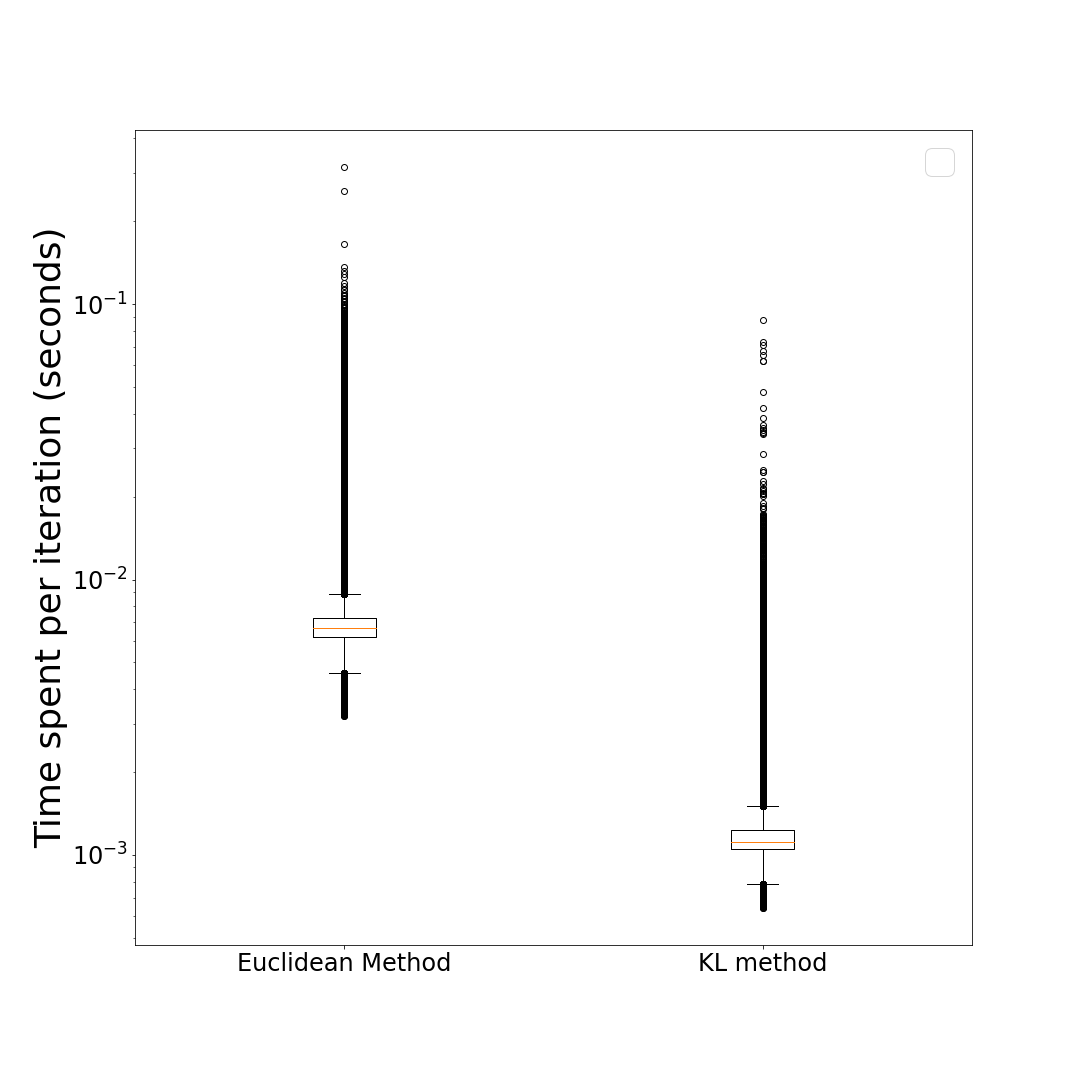}
        \caption{Run time per iteration}
        \label{sfig:1000time}
    \end{subfigure}
    \caption{Matrix game results for $n=1000$}
    \label{fig:matrix1000}
\end{figure*}

\subsection{Gaussian Communication}
We now turn our attention to maximising the information capacity of a noisy \emph{Gaussian communication channel}~\cite[Chapter 9]{thomas2006elements}. In this problem, the goal is to allocate a total power of $P$ across $m$ channels, represented by $p\in\mathbb{R}^m_+$, to maximise the total information capacity of the channels in the presence of allocated noise, represented by $n\in\mathbb{R}^m_+$. The information capacity of the $i$th channel, denoted $C_i(p_i,n_i)$, is the function of power $p_i$ and noise level $n_i$ and is given by
$$C_i(p_i,n_i) = \log\left(1+\frac{\beta_i p_i}{\mu_i+n_i}\right),$$
where $\mu_i>0$ and $\beta_i>0$ are given constants.

Assuming a total power level of $P$ and a total noise level of $N$, optimising for the worst-case scenario by treating the noise allocation as an adversary gives the convex-concave game
$$\max_{p\in\Delta^m_P}\min_{n\in\Delta^m_N} \sum_{i=1}^m C_i(p_i,n_i),$$
where $\Delta_T^n := \{x\in\mathbb{R}^n_+\colon \sum_{i=1}^n x_i = T\}$ is a \textit{scaled} simplex. This problem is also of the form specified in \eqref{eq:saddle} and so can also be formulated as the variational inequality~\eqref{eq:VI}.

The Lipschitz constant of the operator $F$ for this problem is not straightforward to compute, so we apply the adaptive algorithms. Similarly to Section~\ref{subsec:matrix}, we compare the Euclidean and KL versions of Algorithm~\ref{alg:B-aGRAAL}. Since $x\mapsto x\log x$ is $\frac{1}{M}$-strongly convex for $0<x\leq M$, the strong convexity constant of the negative entropy over $\Delta^m_P\times\Delta^m_N$ is $\min\left\{\frac{1}{P},\frac{1}{N}\right\}$. In our experiments, we set $(P,N)=(500,50)$ and generate $\beta\in(0,P]^m$ and $\mu\in(1,N+1]^m$ uniformly. The initial points are chosen as $p_0=\left(\frac{P}{m},\dots,\frac{P}{m}\right)\in\Delta^m_P,n_0=\left(\frac{N}{m},\dots,\frac{N}{M}\right)\in\Delta^m_N$, and $z_0=(p_0,n_0)$, and the initial step-size is taken as $\lambda_0=\frac{\|z_0-\overline{z}_0\|^2}{\|F(z_0)-F(\overline{z}_0)\|^2}$ where $\overline{z}_0$ is again a small random perturbation of $z_0$. It is also worth noting that for the KL version, we had to multiply $\lambda_0$ by a small constant ($10^{-2}$) to avoid numerical instability issues. 

\begin{figure*}[t!]
    \centering
    \begin{subfigure}[t]{0.5\textwidth}
        \centering
        \includegraphics[width=\textwidth]{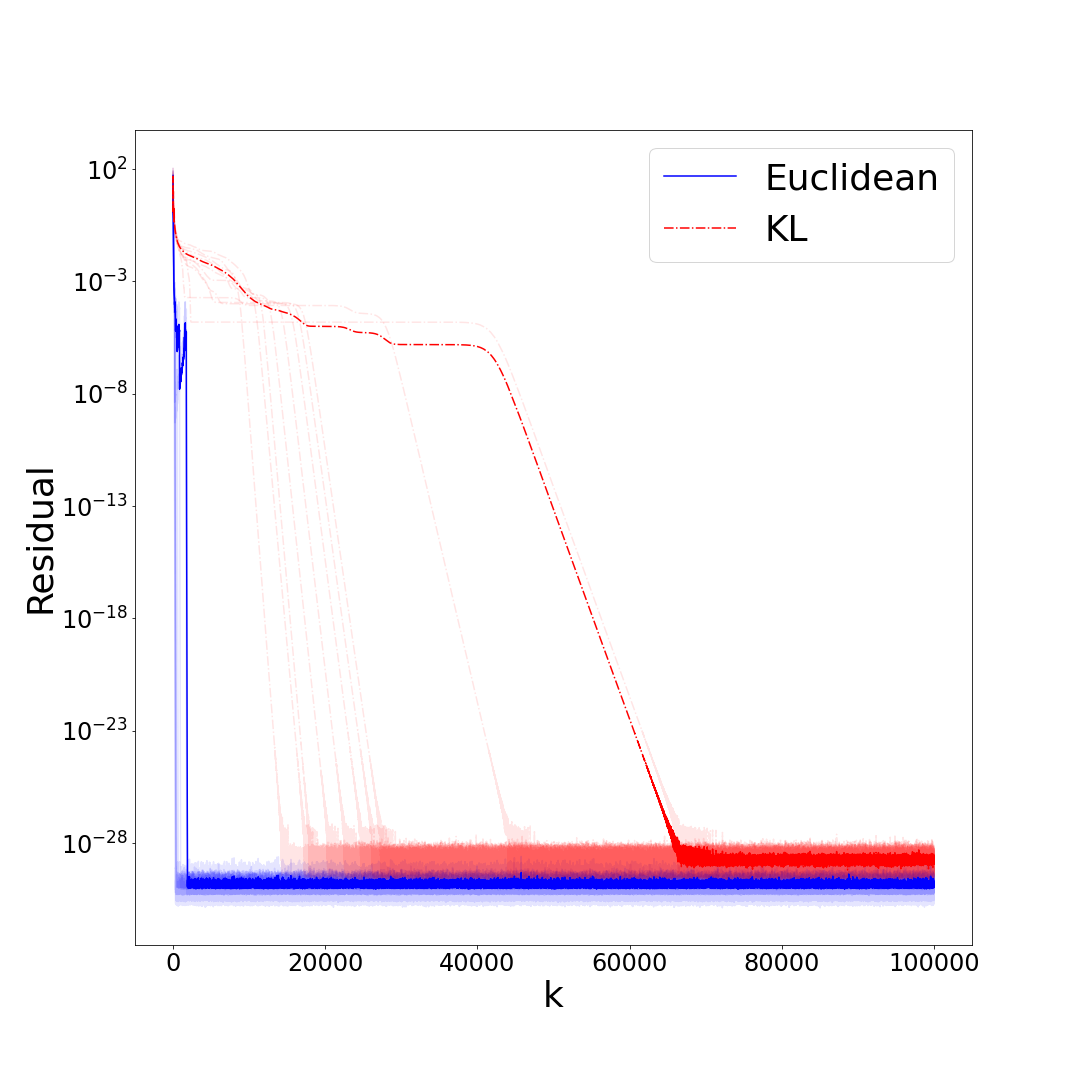}
        \caption{Residual over iterations}
        \label{sfig:100}
    \end{subfigure}%
    ~ 
    \begin{subfigure}[t]{0.5\textwidth}
        \centering
        \includegraphics[width=\textwidth]{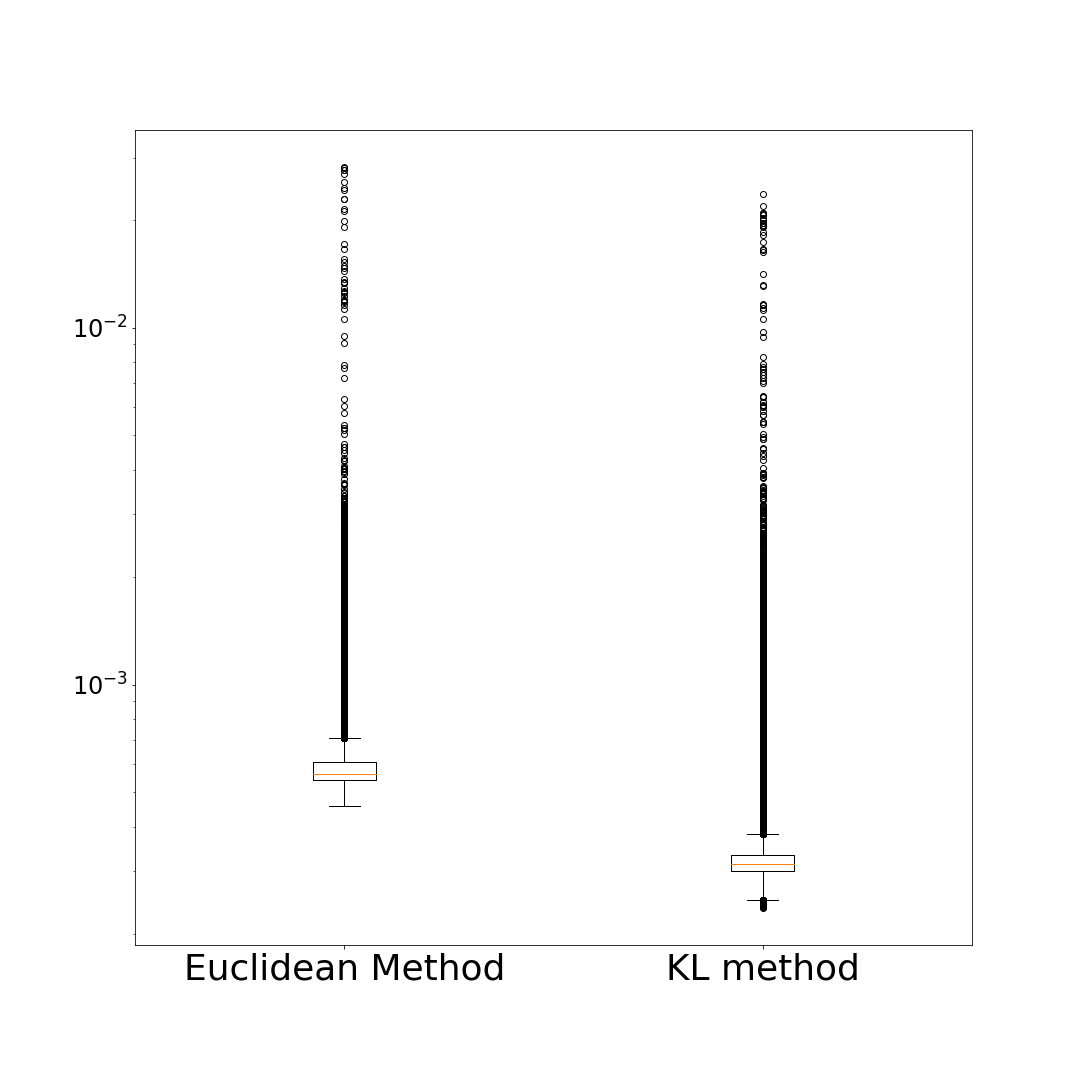}
        \caption{Run time per iteration}
        \label{sfig:100time}
    \end{subfigure}
    \caption{Gaussian communication channel results for $m=100$}
    \label{fig:gauss100}
\end{figure*}
\begin{figure*}[t!]
    \centering
    \begin{subfigure}[t]{0.5\textwidth}
        \centering
        \includegraphics[width=\textwidth]{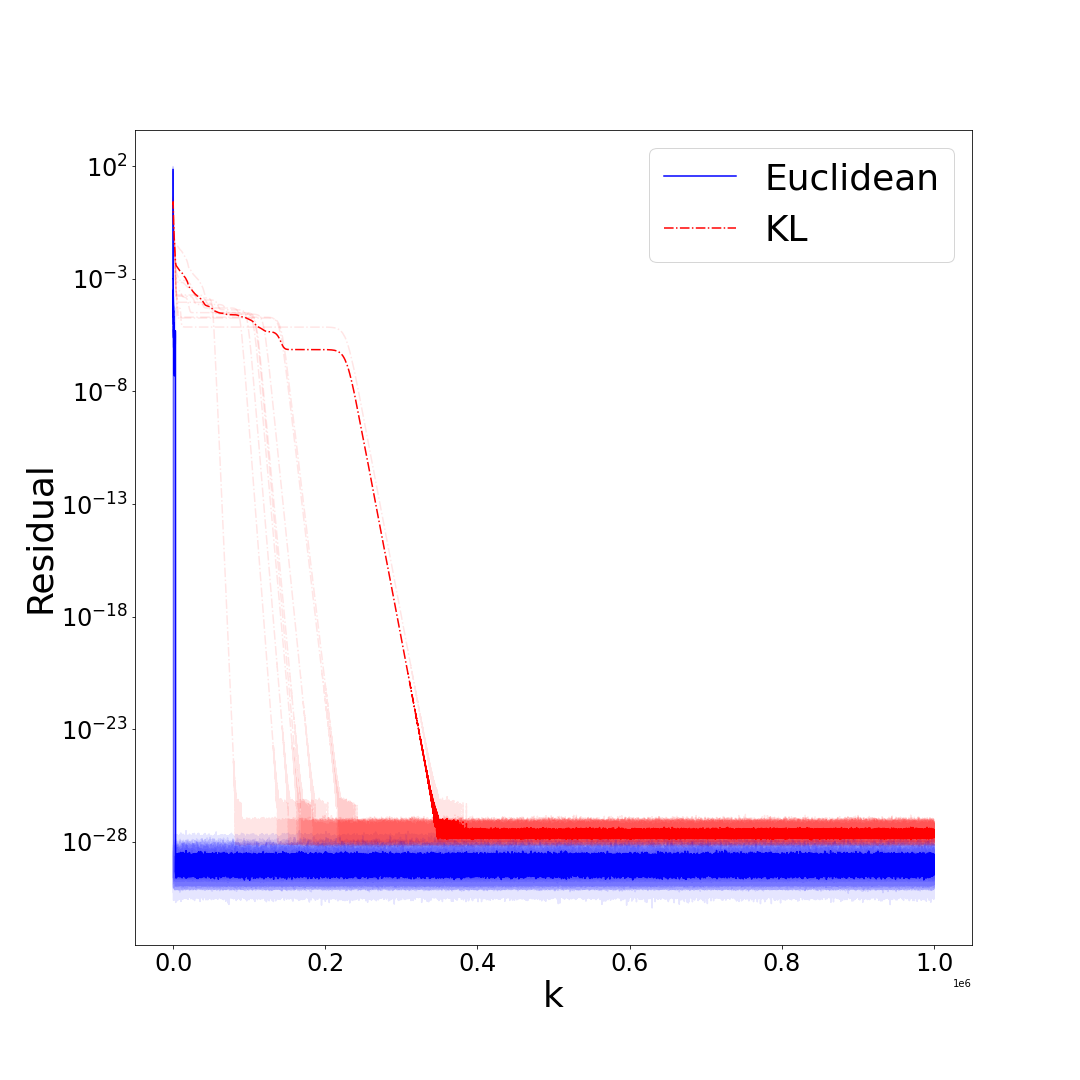}
        \caption{Residual over iterations}
        \label{sfig:200}
    \end{subfigure}%
    ~ 
    \begin{subfigure}[t]{0.5\textwidth}
        \centering
        \includegraphics[width=\textwidth]{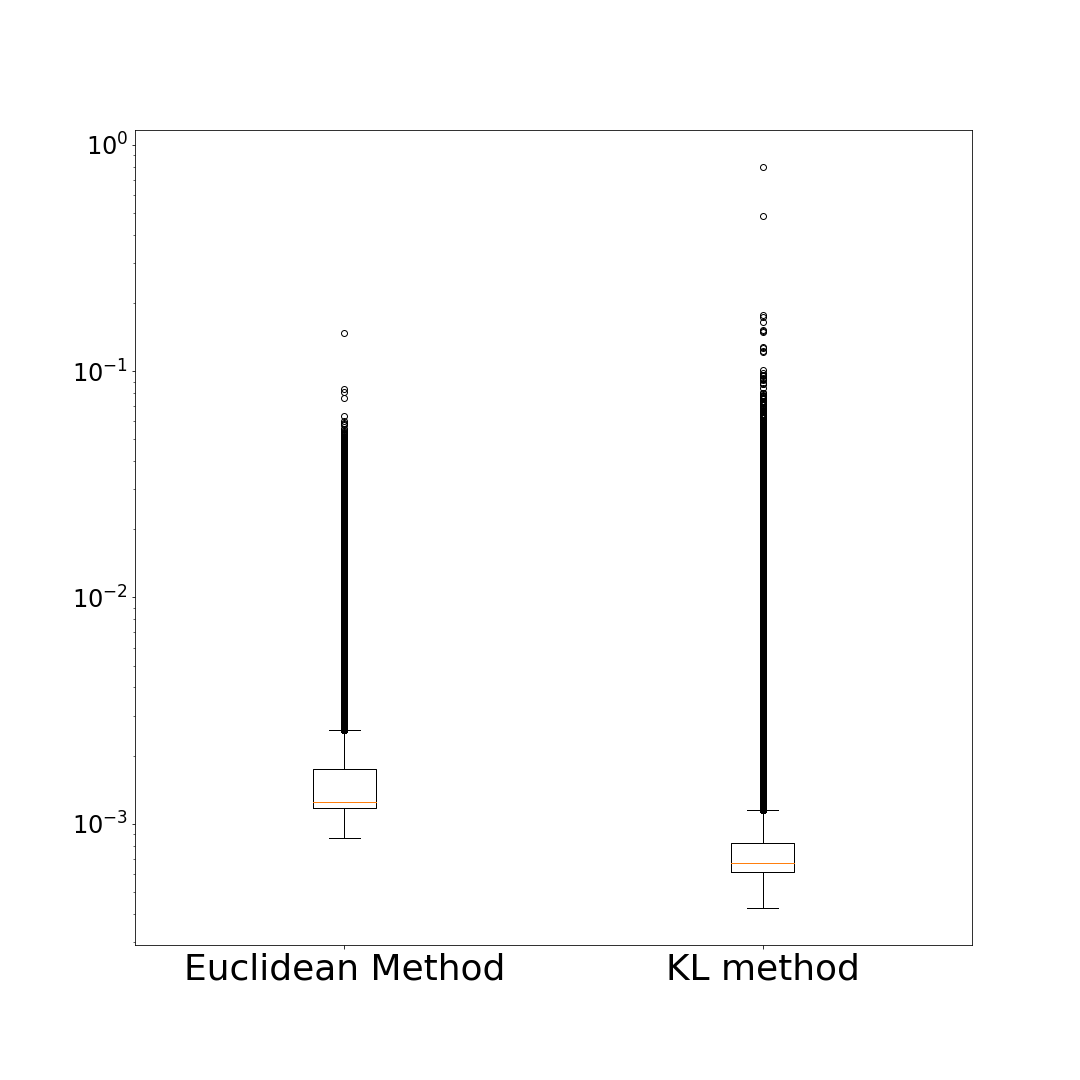}
        \caption{Run time per iteration}
        \label{sfig:200time}
    \end{subfigure}
    \caption{Gaussian communication channel results for $m=200$}
    \label{fig:gauss200}
\end{figure*}

We run two experiments, for $m=100$ and $m=200$, and plot the results in Figures~\ref{fig:gauss100} and~\ref{fig:gauss200} respectively. The KL method is slower than the Euclidean method in terms of the number of iterations and time. However, unlike in the previous section, both methods reach a similar final accuracy.

\subsection{Cournot Completion}
Our final example is a standard $N$-player Cournot oligopoly model \cite[Example 2.1]{bravo2018bandit}. This is a system in which $N$ independent firms supply the market with a quantity of some common good or service. More formally, each firm seeks to maximise their utility subject to their capacity, that is,
\begin{equation}\label{eq:utility}
\displaystyle{\max_{0\leq x_i\leq C_i}} u_i(x_i,x_{-i}) = x_i P(x_T) - c_i x_i,
\end{equation}
where
\begin{itemize}
    \item $x_i\geq 0$ is the quantity of the good supplied by the $i$th firm, $i=1,\dots,N$.
    \item $x_{-i} = (x_1,\dots,x_{i-1},x_{i+1},\dots,x_N)$ is the quantity of the good supplied by all other firms.
    \item $x_T := \sum_{i=1}^N x_i$ is the total amount supplied.
    \item $C_i>0$ is the production capacity of the $i$th firm.
    \item $c_i>0$ is the production cost of the $i$th firm.
    \item $P\colon\mathbb{R}_+\to\mathbb{R}$ is the inverse demand curve.
\end{itemize}
In this section, we consider solutions to this problem in the sense of \emph{Nash equilibria}, which is equivalent to the variational inequality~\eqref{eq:VI2} with 
$$F = \left(-\frac{\partial u_1}{\partial x_1}, \dots, -\frac{\partial u_N}{\partial x_N}\right),\quad K = \left[0,C_1\right]\times\dots\times\left[0,C_N\right],\quad g = \iota_K.$$
By choosing a function $h$ such that $\dom h=K$, it is possible to implicitly enforce the capacity constraint in this problem and avoid performing projections. Two examples, over a single closed interval $[\alpha,\beta]$, present themselves which satisfy our assumptions:
$$ h_1(x) = (x-\alpha)\log(x-\alpha) + (\beta-x)\log(\beta-x), \quad
    h_2(x) = -\sqrt{(x-\alpha)(\beta-x)}.$$
When $\alpha=0$ and $\beta=1$, $h_1$ is called the \emph{Fermi-Dirac Entropy}, and similarly when $\alpha=-1$ and $\beta=1$, $D_{h_2}$ is called the \emph{Hellinger distance} (\cite[Example 2.2]{bauschke2019linear}, \cite[Example 1]{bauschke2017descent}, \cite[Example 2.1]{teboulle2018simplified}). Then we sum over the independent functions of one variable, with appropriately set intervals $\alpha=0,\beta=C_i$, to create the Bregman distance over the closed box $K$. We also compute the strong convexity constants by minimising $\nabla^2 h_1$ over $(\alpha,\beta)$, which gives $\sigma = \frac{4}{\beta-\alpha}$, and similarly for $h_2$, $\sigma=\frac{2}{\beta-\alpha}$.

In our experiments, $P$ is taken to be a linear inverse demand curve given by $P(x) = a-bx$ for $a,b>0$. All parameters are generated by a log-normal distribution, except the cost vector $c\in\mathbb{R}^N$, which is generated uniformly in $\left[\frac{C_1}{100},\frac{C_1}{5}\right]\times\dots\times\left[\frac{C_N}{100},\frac{C_N}{5}\right]$. We run two experiments, with $N=2000$ and $N=5000$, the results of which are shown in Figures~\ref{fig:nash2000} and~\ref{fig:nash5000}.  The initial points are chosen as $z_0 = \frac{C}{2}$, and $\lambda_0=1$.
\begin{figure*}[t!]
    \centering
    \begin{subfigure}[t]{0.5\textwidth}
        \centering
        \includegraphics[width=\textwidth]{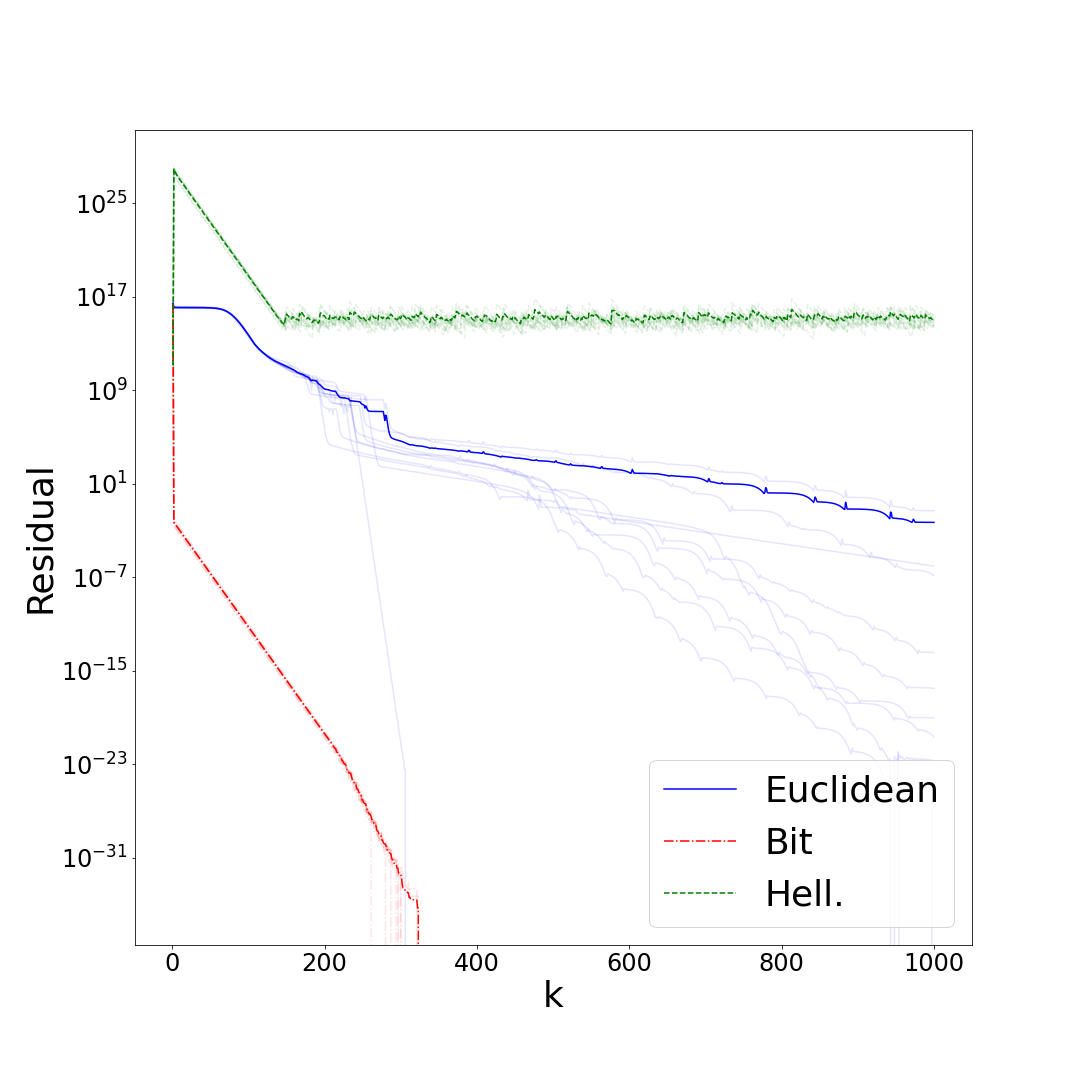}
        \caption{Residual over iterations}
    \end{subfigure}%
    ~ 
    \begin{subfigure}[t]{0.5\textwidth}
        \centering
        \includegraphics[width=\textwidth]{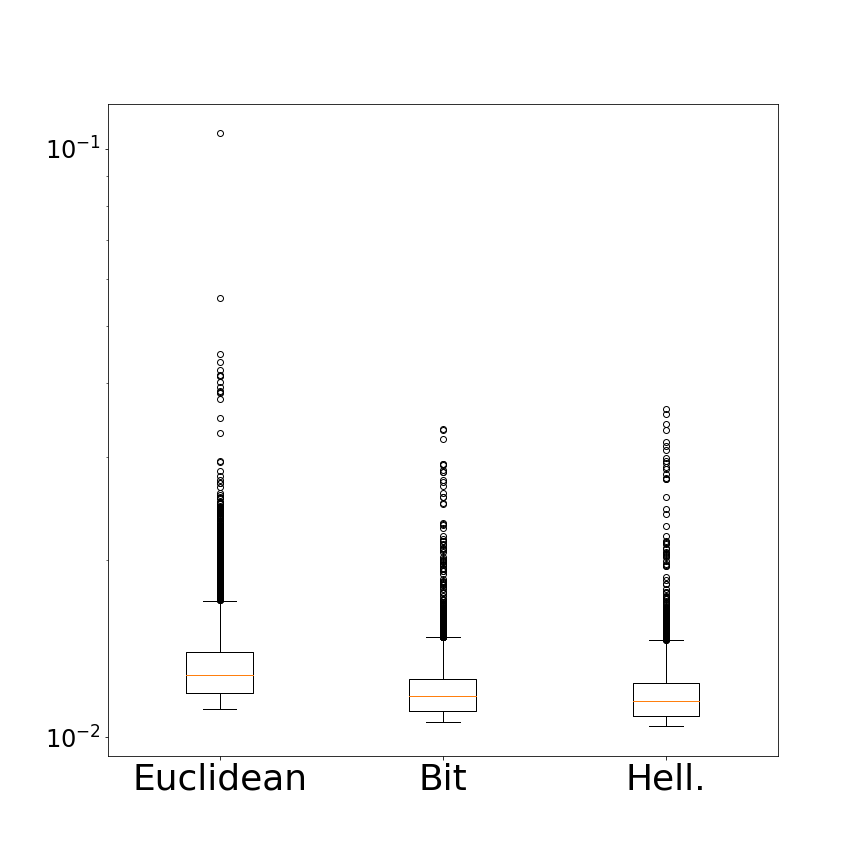}
        \caption{Time per iteration}\label{sfig:2000time}
    \end{subfigure}
    \caption{Results for $N=2000$}
    \label{fig:nash2000}
\end{figure*}
\begin{figure*}[t!]
    \centering
    \begin{subfigure}[t]{0.5\textwidth}
        \centering
        \includegraphics[width=\textwidth]{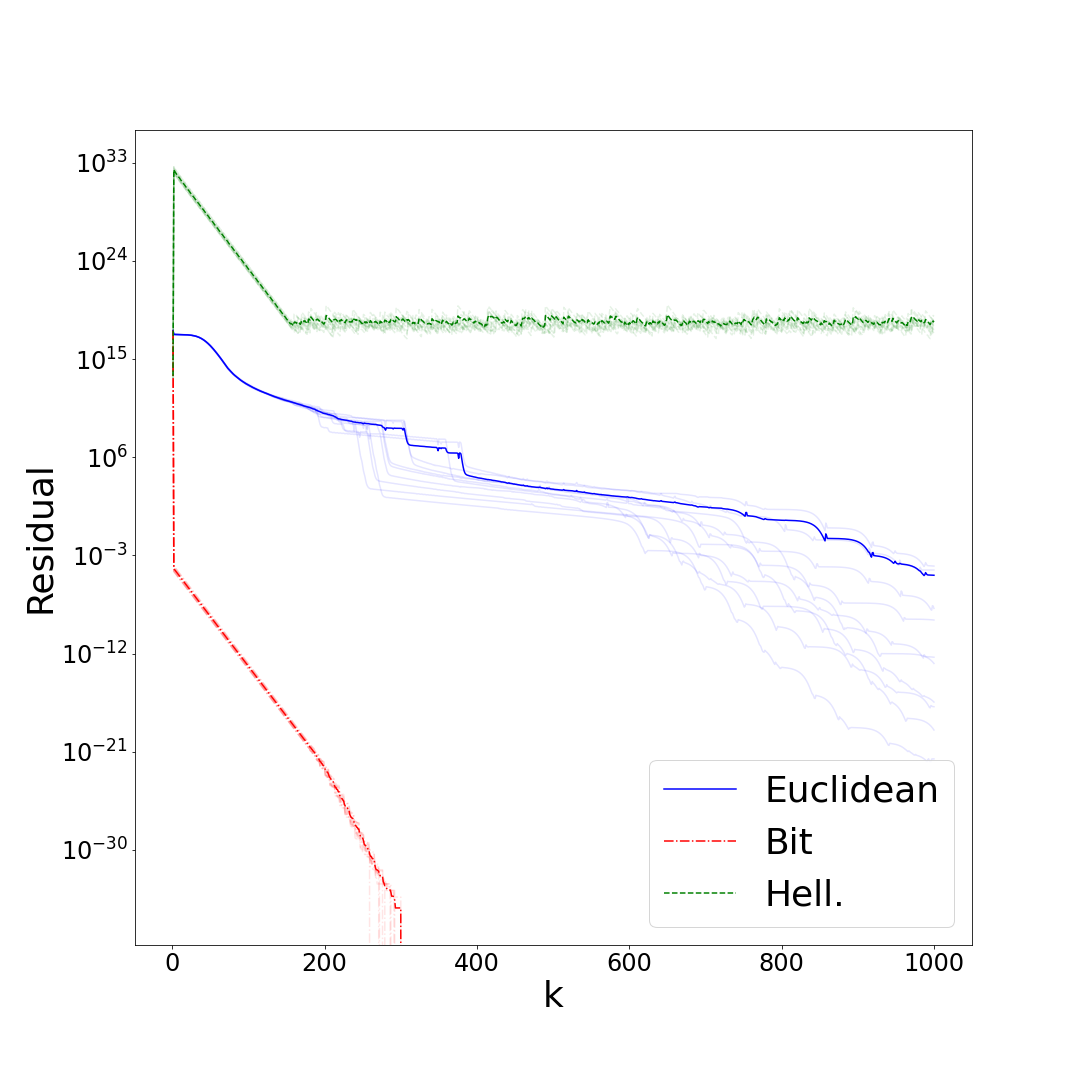}
        \caption{Residual over iterations}
    \end{subfigure}%
    ~ 
    \begin{subfigure}[t]{0.5\textwidth}
        \centering
        \includegraphics[width=\textwidth]{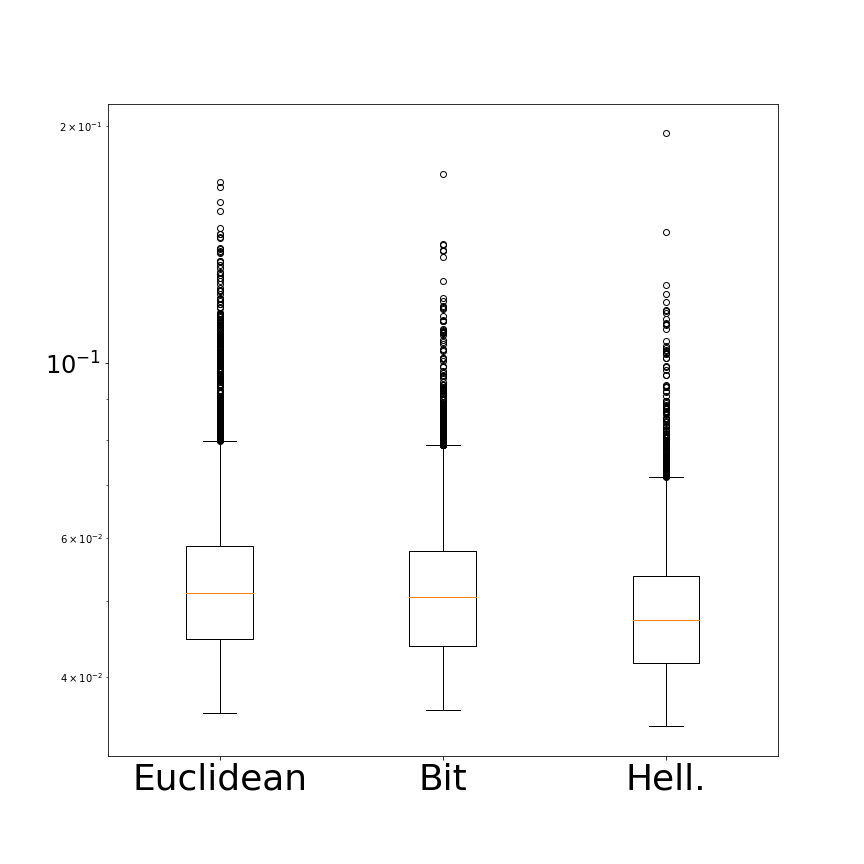}
        \caption{Time per iteration}\label{sfig:5000time}
    \end{subfigure}
    \caption{Results for $N=5000$}
    \label{fig:nash5000}
\end{figure*}

We observe here that the final accuracy depends heavily on the choice of Bregman distance. In both figures, we observe that the Hellinger method is not making any progress after approximately $200$ iterations, while the Euclidean method acheives a modest final accuracy. Meanwhile the Bit method is very fast and accurate, converging to a near $0$ tolerance in roughly $300$ iterations in all instances. Finally, we note that all methods are roughly equal in terms of time per iteration. This is to be expected, since the Euclidean method requires a projection which evaluates $\min\{\max\{x_i,0\}, C_i\}$ for each component, whereas the Bregman methods require evaluating $\nabla h,(\nabla h)^{-1}$ over each component --- all of which are taking $O(N)$ time.

\section{Conclusion}\label{sec:conc}

In this paper, we extended the adaptive method aGRAAL (Algorithm~\ref{alg:aGRAAL}) to the Bregman distance setting. We proposed two such extensions: The first, Algorithm \ref{alg:B-GRAAL}, generalises the fixed step-size GRAAL and converges under the same assumptions for a strongly-convex Bregman function $h\colon\Hilbert\to\mathbb{R}$.  The second, Algorithm \ref{alg:B-aGRAAL}, generalises Algorithm \ref{alg:aGRAAL}, and converges in a more restrictive setting. We first examined the performance of Algorithm~\ref{alg:B-GRAAL}, and found that the KL version is less favourable than the Euclidean version, despite the reduced time per iteration. We then tested Algorithm~\ref{alg:B-aGRAAL} on a convex-concave game for Gaussian communication channels, where our new method performed worse with respect to the KL divergence when compared to the Euclidean method, although the run-time per iteration was again significantly shorter as was expected. Finally, we examined a Cournot completion model, where one of Bregman based methods reached a much higher accuracy very quickly.

We conclude by outlining directions for further research:
\begin{itemize}
    \item It would be interesting to know whether Algorithm~\ref{alg:B-aGRAAL} can be shown to converge in a more general setting than what we have shown, and if so, under what circumstances. The difficulties in our analysis arose from two issues: first, the estimate derived in~\eqref{eq:ad distance} for the Bregman case is weaker than the Euclidean equality $\|z_{k+1}-\overline{z}_{k+1}\|^2 = \frac{1}{\phi^2}\|z_{k+1}-\overline{z}_k\|^2$ used in~\cite{malitsky2020golden}; and second, the inability to bound $\theta_k$ below without additional assumptions (as the bound in Lemma~\ref{lem:bounded} can be arbitrarily small in general).
    \item Throughout this paper, $h$ was assumed strongly convex, but whether or not this assumption can be relaxed is unclear. One potential replacement for strong convexity is considered in \cite{bauschkedykstra, laude2020bregman} where $h$ is twice differentiable such that the Hessian matrix $\nabla^2 h(z)$ is positive definite for all $z\in\intr\dom h$. Within the scope of twice differentiable functions, such a condition lies between strict and strong convexity, the main consequence being that $\nabla h$ is locally Lipschitz and $h$ is locally strongly convex. Indeed, for $\sigma$-strongly convex $h$ with $\eta$-Lipschitz gradient, one can derive the estimates
    $$\frac{\sigma}{2}\|z-z'\|^2 \leq D_h(z,z') \leq \frac{\eta}{2}\|z-z'\|^2\quad\forall z,z'.$$
    If such inequalities hold on a local scale, then it would remain to be seen whether these coefficients can be estimated and used in the same way that $\lambda_k$ approximates an inverse of the local Lipschitz constant of $F$.
    \item In the context of the convex composite optimisation problems of the form $\min_{x\in \Hilbert}f(x) + g(x)$ where $g$ is nonsmooth and $f$ is smooth, the Bregman proximal gradient algorithm \cite{teboulle2018simplified, bauschke2017descent} is known to converge in a more general setting than Lipschitz continuity. Specifically, $L$-Lipschitz continuity of $\nabla f$ can be relaxed to convexity of the function $Lh-f$, which indeed holds if $\nabla f$ is Lipschitz and $h$ is strongly-convex. It would be interesting to see if a similar relaxation of Lipschitz continuous can be used in the context of the algorithms discussed here.
\end{itemize}

\paragraph{Acknowledgements.}
DJU and MKT are supported in part by Australian Research Council grant DE200100063.



\end{document}